\documentclass[12pt,letterpaper,reqno]{article}
\usepackage{graphicx}
\usepackage{amsmath}
\usepackage{amsfonts}
\usepackage{amsthm}
\usepackage{amssymb}
\usepackage[T1]{fontenc}
\usepackage[toc,page]{appendix}
\usepackage{verbatim}

\usepackage{hyperref}

\theoremstyle{plain}
\newtheorem{theorem}{Theorem}[section]

\newtheorem{condition}[theorem]{Assumption}

\newtheorem{corollary}[theorem]{Corollary}

\newtheorem{definition}[theorem]{Definition}
\newtheorem{example}[theorem]{Example}

\newtheorem{lemma}[theorem]{Lemma}

\newtheorem{remark}[theorem]{Remark}

\usepackage{ulem,soul,xcolor}

\newcommand\repd[2]{\sout{\color{red} #1}{\color{blue} #2}}

\newcommand{\mr}{\mathbb{R}}
\newcommand{\mc}{\mathcal}

\newcommand{\lan}{\langle}
\newcommand{\ran}{\rangle}
\newcommand{\diverg}{\mathrm{div}}
\newcommand{\tr}{\mathrm{tr}}
\newcommand{\mt}{\mathbb{T}}
\newcommand{\eps}{\varepsilon}
\newcommand{\rmd}{\,\mathrm{d}}
\newcommand{\E}{\mathbb{E}}
\newcommand{\td}{\tilde}
\newcommand{\curl}{\mathrm{curl}\,}
\newcommand{\divv}{\mathrm{div}\,}

\newcommand{\nH}{H}
\newcommand{\Radon}{Borel }

\newcommand\wek[2]{{\begin{pmatrix}{#1}\\{#2}\end{pmatrix}}}
\newcommand\mat[4]{{
\begin{pmatrix}#1&#2 \\#3 & #4\end{pmatrix}
}}

\newcommand\dela[1]{}

\newcommand{\rH}{\mathrm{H}}

\newcommand{\lb}{\langle}
\newcommand{\rb}{\rangle}

\numberwithin{equation}{section}

\title{Existence for stochastic 2D Euler equations with positive $H^{-1}$ vorticity}
\author{
	Zdzis\l aw Brze\'zniak\footnote{Department of Mathematics, University of York, YO10 5DD Heslington, York, UK. E-mail address: {\tt zdzislaw.brzezniak@york.ac.uk}}\;\; and
	Mario Maurelli\footnote{Dipartimento di Matematica `Federigo Enriques', Universit\`a degli Studi di Milano, via Saldini 50, 20133 Milano, Italy. E-mail address: {\tt mario.maurelli@unimi.it}}}
\date{}

\begin{document}

\maketitle

\begin{abstract}
We prove the existence of non-negative measure and $H^{-1}$-valued vorticity solutions to the stochastic 2D Euler equations in the vorticity form  with the transport type noise, starting from any non-negative vortex sheet. This extends the result by Delort \cite{Del1991} to the stochastic case.
\end{abstract}

\tableofcontents

\section{Introduction}

In this paper we consider the two-dimensional  stochastic Euler equations, in the vorticity form, with the transport type gaussian noise with the periodic boundary conditions. In other words we study the following stochastic  evolution equation on the two-dimensional torus $\mathbb{T}^2$:
\begin{align}
\begin{aligned}\label{eqn: Euler stoch vorticity form-Intro}
&\partial_t \xi + u\cdot \nabla \xi + \sum_k \sigma_k\cdot \nabla \xi \circ \,\,\dot{W}^k =0,\\
&u = K \ast \xi,
\end{aligned}
\end{align}
see below for explanation of the notation used here.

This equation represents the motion of an incompressible fluid in a periodic domain, perturbed by a noise of the transport-type. In our main result, Theorem \ref{thm:main}, we prove the weak existence,  in both the probabilistic and the analytic sense,  of a certain measure-valued solution for this equation.  More precisely, we  prove that for any initial datum $\xi_0 \in H^{-1}$ which is also a  non-negative measure, there exists a weak, measure and $H^{-1}$-valued solution to \eqref{eqn: Euler stoch vorticity form-Intro}. This result extends the existence result by Delort \cite{Del1991} from the deterministic to the stochastic case.

The deterministic incompressible, $d$-dimensional Euler equations written in the following form
\begin{align}
\begin{aligned}\label{eqn-Euler}
&\partial_t u +(u\cdot\nabla) u +\nabla p = 0,\\
&\diverg u =0,
\end{aligned}
\end{align}
describe the evolution of the velocity field  $u(t,x)\in \mr^d$ and the scalar pressure $p(t,x)\in \mr$ of an incompressible fluid.  For the simplicity sake we assume that $x$ belongs to $d$-dimensional torus  $\mt^d$. The local (in time) well-posedness of \eqref{eqn-Euler} withing a class of  smooth solutions has been known    since at l;east of the seminal work  \cite{EbiMar1970}.  A general description of the Euler equations can be found in many excellent works, for example in  \cite{MarPul1994}, \cite{Lio1996}, \cite{MajBer2002}. In the two-dimensional case, i.e. when $d=2$, the Euler equations \eqref{eqn-Euler} are known to be, at a heuristic level,  equivalent to the following  non-linear transport type equation, called  the vorticity equation,
\begin{align}
\begin{aligned}\label{eq:Eulervort_intro}
&\partial_t \xi + u\cdot \nabla \xi =0,\\
&u(t) = K \ast \xi(t), \;\;.
\end{aligned}
\end{align}
Here $\xi(t,x):=\curl  u(t,x) = \partial_{x_1}u^2(t,x) -\partial_{x_2}u^1(t,x)$ is the scalar vorticity of the fluid, which describes  how fast the fluid rotates around a point $x$.  The kernel $K:\mt^2\to\mr^2$ is given by $K=\nabla^\perp G$, where $G:\mt^2\to \mr$ is the Green function of the Laplacian on $L^2(\mt^2)$, restricted to   mean zero functions, see \cite{BrzFlaMau2016}. Alternatively, $u(t) =  \nabla^\perp \Delta_0^{-1} \xi(t)$, $t\geq 0$.

 The global well-posedness within the class of  essentially bounded solutions  has been proved in \cite{Wol1933} and \cite{Yud1963}. See also \cite{Yud1995} and \cite{Serfati1994} for an extension to almost bounded functions and \cite[Section 2.3]{MarPul1994} for a different proof using flows.  In the case of deterministic 2D Euler Equations,  there is a huge number of papers using the vorticity formulation
\eqref{eq:Eulervort_intro}  of these equations with the initial data belonging to various functional, e.g. Besov,  spaces, see for instance Serfati \cite{Serfati1994,Serfati1995},  Vishik \cite{Vishik1999}, Bernicot et al \cite{Bernicot2016}, Bahouri and Chemin \cite{BahChe1994}.

In larger classes than of  bounded solutions, the global well-posedness  has been established
by Majda and Di Perna in \cite{DiPMaj1987} for the initial  vorticity belonging to the spaces  $L^p(\mt^2)$, with $1<p<\infty$.  To be more precise, the result in \cite{DiPMaj1987}  is given in  the full euclidean space $\mathbb{R}^2$, but for the simplicity sake, in this introduction we will not distinguish between the cases of the full space and the torus. The case  we are interested in is of the initial  the vorticity being a  measure and the solution is measure-valued. One of the main results in this context is the global existence result by Delort \cite{Del1991}, where the vorticity is assumed to have  a distinguished sign and to be an element of the Sobolev space $H^{-1}(\mt^2)$. More  precisely, Delort proved that   for any  initial datum $\xi_0 \in H^{-1}(\mt^2)$ which is also a non-negative measure, there exists a non-negative measure- and $H^{-1}(\mt^2)$-valued solution $\xi$ to the vorticity equation \eqref{eq:Eulervort_intro}  such that $\xi(0)=\xi_0$. Delort's result  includes the case of an initial vortex sheet, that is when $\xi_0$ is concentrated on a $C^1$ curve. Later papers by Schochet \cite{Sch1995} and Poupaud \cite{Pou2002} gave a somehow clearer argument which we will  mostly use  in the rest of the paper. A more general existence result, where the vorticity is a sum of a non-negative measure $H^{-1}(\mt^2)$ and an $L^p(\mt^2)$ function, for some $p\ge 1$, has been studied  in \cite{VecWu1993}; \cite{Serfati1998,Serfati1998b} show the short-time existence when the initial vortex sheet consists of two parts with opposite sign and disjoint supports, and establish non-trivial bounds on propagation of the support of solutions. Studying  of irregular vorticity solutions has several motivations.  Firstly, such solutions  represent physically relevant situations, where the vorticity is concentrated on sets of   Lebesgue measure zero, see e.g.~\cite[Chapter 6]{MarPul1994}.  Secondly, such solutions are   related to the possible anomalous energy dissipation; to the best of our knowledge, the energy conservation or dissipation remains an open problem for class of the Delort solutions, see \cite{CLNS2016} for energy conservation for unbounded vorticity. Thirdly, but not lastly, they are also related  to the boundary layers phenomena, see e.g.~\cite{Cho1978}. However we do not pursue  these topics here. Let us only briefly discuss the two other classes of solutions. The authors of \cite{DiPMaj1987_2} introduce another (and different) concept of a ``measure-valued solution''.  This concept refers to a different, weaker type of a solution, i.e. in the sense of the Young measures, which we do not consider here in this paper. An interested reader can consult the paper \cite{Majda_1993} by Majda for more details about comparison between the framework of \cite{Del1991} and that of \cite{DiPMaj1987_2}. Recently a variety of types of weak solutions to \eqref{eqn-Euler}, where $\xi$ does not need to be a function or measure, have been constructed using methods related to convex integration, showing the non-uniqueness in certain functional spaces, as for instance $u\in L^\infty$, and energy dissipation, see e.g. \cite{DeLSze2009,Ise2018}. In particular, \cite{Sze2011} exhibits infinitely many weak solutions to \eqref{eqn-Euler} with an initial condition $\xi_0$ being a non-negative $H^{-1}$ measure. However, the author claims that, for those weak solutions, the vorticity $\xi = \curl u$ is in general not measure-valued. We are not aware whether one can construct wild solutions to \eqref{eqn-Euler} by convex integration methods whose vorticity is measure-valued, see also the discussion in \cite{IftLopNus2020}.


Before moving on  to the stochastic case, let us present  a short idea of the proof of the Delort result. The strategy passes through an usual compactness and convergence argument.  Take a sequence of approximants for which we can  show uniform a'priori  bounds, prove  the compactness and  show the stability of the equation in the limit. The main problem comes from the stability step, in particular from the proof of the stability of the non-linear term. The main idea used  by Delort, and later also by Schochet \cite{Sch1995},  is that the non-linearity  is continuous with respect to  non-negative, or more-generally bounded from below, $H^{-1}(\mt^2)$ measures, in the weak topologies. The precise description of the topologies involved will be given later. Hence it suffices to show the conservation  of the non-negativity, or at least a bound from below, and the uniform a'priori bounds on the total variation  and the $H^{-1}(\mt^2)$ norms of the approximations of the  solution. The transport and divergence-free nature of the vorticity equation \eqref{eq:Eulervort_intro} implies easily the non-negativity and the a priori bound on the total variation norm. Since the  $H^{-1}(\mt^2)$ norm of $\xi$ is equivalent to the $L^2(\mt^2)$ norm of $u=K \ast \xi$, i.e.  the energy, so the a'priori $H^{-1}(\mt^2)$ bound is equivalent to an a priori bound on the energy of $u$, which is also classical and available. This allows to have compactness and stability as required.

The stochastic  Euler Equations perturbed   by gaussian noise have been investigated in a large number  of publications.  Here we only give a review of a small selection of them. Many articles study such equations with either   additive or multiplicative noise but dependent on $u$ and not on its gradient. The first works are \cite{BesFla1999}, which shows  the global well-posedness of strong solutions in the case of   additive noise, for bounded vorticity in a 2D domain (though adaptedness is not proven there), \cite{Bes1999} and \cite{BrzPes2001}, which show the global existence of martingale solutions, with multiplicative noise, for $L^2(\mt^2)$ and, in the second work, $L^p(\mt^2)$ vorticity also in the  2D domains. Let us  also mention the paper \cite{MikVal2000} for a geometric approach in the case of finite-dimensional additive noise and \cite{CapCut1999} for an approach via nonstandard analysis. The works \cite{Kim2009}, \cite{GlaVic2014} contain proofs of a  local strong well-posedness of  smooth solutions, with additive and, in the second work, multiplicative noise, in general dimensions. The latter paper  contains also a proof of  a  global well-posedness  of  smooth solutions in 2D for additive or   multiplicative linear noise. The work \cite{Kim2015} deals with  low regularity solutions to the stochastic Euler equations driven by an additive noise, in any space dimension. More precisely,
the authour shows the existence of the s measure-valued solutions, in the sense of Young measures as in \cite{DiPMaj1987_2}, see the discussion above,  for the initial vorticity in the $H^{-1}$ space.

The transport type noise for the stochastic Navier-Stokes Equations has been first introduced in early papers \cite{Brz+Cap+Fland_1991,Brz+Cap+Fland_1992} and later in, e.g. \cite{Mik+Roz_2004} and \cite{BrzMot2013}.
In the case of the stochastic Euler Equations, the transport type noise has been first
 studied in \cite{Yok2014}, \cite{StaYok2014}, \cite{CruTor2015}, where the transport term acts  on the velocity and not on the vorticity.  These works show the global existence of martingale solutions for $L^2(\mt^2)$ vorticity in 2D domains. The model \eqref{eqn: Euler stoch vorticity form-Intro} has been introduced  in \cite{BrzFlaMau2016}, where the authours  prove the  global strong well-posedness in the class of  bounded vorticity solutions on the two-dimensional torus. In the recent papers \cite{FlaLuo2019} and \cite{FlaLuo2019_2}, the authors consider  the same model, but with very irregular vorticity, not in $H^{-1}(\mt^2)$ as in the present study, and  show the  existence of solutions for random  initial data.
   An analogue of the model \eqref{eqn: Euler stoch vorticity form-Intro} in the 3D case, where a stochastic advection term also appears, has been recently  considered in \cite{CriFlaHol2019}, where the authors  show the local strong well-posedness and a  Beale-Kato-Majda type criterion for the non-explosion. Let us  also point out the work \cite{FlaGubPri2011}, in which the authours consider equation  \eqref{eqn: Euler stoch vorticity form-Intro} with the initial mass concentrated in a finite number of point vortices and shows a regularization by noise phenomenon, precisely that, with full probability, no collapse of vortices happens with a certain noise, while it does happen without noise.

There are several reasons why we use transport noise. One of them, maybe the most important one,  is that the equations driven by   a transport type noise preserve the transport structure of the vorticity equation.  At a formal level, a solution $\xi(t,x,\omega)$ follows the characteristics of the associated SDEs, i.e.  $\xi(t,X(t,x,\omega),\omega)=\xi(0,x)$, where $X(t,x,\omega)$ is the stochastic flow solution to the following SDEs
\begin{align*}
dX(t) = u(t,X(t))dt +\sum_k\sigma_k(X(t)) \circ dW^k(t).
\end{align*}
This fact can be proven by   the It\^o formula.   The Stratonovich form of the noise is essential, because the It\^o formula for this noise works as the chain rule, without second-order corrections. Furthermore, there is a derivation of models with stochastic transport term in \cite{Hol2015}, \cite{DriHol2018}, there are applications using the transport noise to model uncertainties, e.g.~\cite{CCHPS2018}, also the linear stochastic transport equation has been used as a toy model for turbulence, see e.g.~\cite{Gaw2008}, though we will not go into these directions here. The main feature of interest here is the fact that the transport noise preserves,  at least the heuristic level,  the $L^\infty(\mt^2)$ norm of the solution and also, when the coefficients $\sigma_k$ are divergence-free as here, the $L^p(\mt^2)$ norm for any $1\le p\le\infty$ (this can be seen for example via a priori bounds for the $L^p$ norm). As a consequence, the transport noise preserves the total mass of the vorticity and its positivity\footnote{ i.e. if $\xi_0\ge0$, then also $\xi_t\ge0$,} two properties that are crucial to apply the Schochet's argument \cite{Sch1995}. This is not the case for a generic additive or multiplicative noise as considered in the above mentioned papers. Indeed the additive noise does not guarantee positivity of the solution, while linear multiplicative noise does not guarantee mass conservation. We should admit though that additive  or certain multiplicative noise could be used for our purposes. Indeed what we really need to have the  Delort's result in the stochastic case  are  some uniform $\mc{M}(\mt^2)$ bounds of the vorticity and some uniform $L^\infty$ bounds, or even $L^p$ bounds in the line of the more general result in \cite{VecWu1993}, on the negative part of the vorticity. Now a linear multiplicative noise acting on  on the vorticity of the form $\sigma_k\xi \circ \,\,\dot{W}^k$ preserves the positivity and may also give some uniform bound on the mass; for additive noise, a bound from below of the vorticity may be proved via a transformation \`a la Doss-Sussmann from the SPDE to a random PDE.

Let us also  remark here that the stochastic vorticity equation \eqref{eqn: Euler stoch vorticity form-Intro} preserves  the enstrophy, that is the $L^2(\mt^2)$ norm of the vorticity solution $\xi$, but not the energy, that is the $L^2(\mt^2)$ norm of the velocity $u=K \ast \xi$. Indeed, the equation for the velocity field $u(t,x,\omega)$ can be written, in a purely  formal way,  as
\begin{align}
\begin{aligned}\label{eq:stochEulervel_intro}
&\partial_t u +(u\cdot\nabla) u +\sum_k(\sigma_k\cdot\nabla +D\sigma_k)u \circ \,\,\dot{W}^k = -\nabla p -\gamma,\\
&\divv u=0,
\end{aligned}
\end{align}
with unknown   pressure $p(t,x,\omega)$ and a time dependent constant $\gamma(t,\omega)$. The constant  $\gamma$
 is needed here in order  to keep $u$ with zero spatial mean.\footnote{Thus, as the pressure is often viewed as the Lagrange multiplier corresponding to the divergence zero constraint, the constant $\gamma$ could be seen as the
  the Lagrange multiplier corresponding to the  mean  zero constraint.}

 In the above  equation \eqref{eq:stochEulervel_intro}, a new  zero order term $(D\sigma_k) u \circ \,\,\dot{W}^k$  causes the $L^2$-norm of the velocity not to be preserved anymore.

Let us also mention that the transport noise has been used to show the so called regularization by noise phenomena for the stochastic  transport equations, see \cite{FlaGubPri2010} and \cite{FedFla2013} among many others.  Such results have been so far mostly limited to the linear cases, with a few positive results for other nonlinear hyperbolic cases, see e.g.~\cite{DelFlaVin2014}.
A  generalisation of the results from the paper   \cite{FlaGubPri2011} to the case of a more general, measure-valued vorticity provides new and challenging  difficulties.

Before embarking to describing contributions of our paper, let us  point out two peculiarities of the strategy used, in the stochastic case, in order to prove the existence of a  martingale solution, see for example \cite{BrzPes2001}. Analogously to the deterministic case, the strategy passes through the usual tightness and convergence argument. One  takes a sequence of approximants, shows a'priori uniform bounds for suitable moments and deduces  the tightness of the laws, then one  shows the stability of the equation in the limit. We want to emphasize  two facts, which we will use in our proof as well. Firstly, to prove the  tightness of the laws, say, in $C([0,T];X)$, where $X$ is a suitable  functional space on the torus $\mt^2$, one needs a uniform bound on the marginals and also a uniform bound in $C^\alpha([0,T];Y)$, where $Y$ is another functional space, typically a negative-order Sobolev space, containing $X$. The latter is the stochastic counterpart of the Aubin-Lions Lemma, used  since at least the paper  \cite{FlaGat1995}. Secondly, for the stability, one option is to exploit the Skorokhod representation theorem which allows one to pass from convergence in law to a.s.~convergence in the aforementioned space $C([0,T];X)$.

Now we are ready to describe in detail  our paper. The main result, i.e.  Theorem \ref{thm:main} extends the result by Delort \cite{Del1991} to  the stochastic setting with the transport type noise. To the best of our knowledge, no such extension has been studied before. Our proof combines the argument by Schochet \cite{Sch1995} to deal with the nonlinear term,   with the argument by e.g. \cite{BrzPes2001} to deal with the stochasticity, as explained before.  More details are given in Seciton \ref{sec:main_strategy}.

Two comments are in place. Firstly, as mentioned earlier, the transport type noise is responsible for the conservation  of mass and of positivity and thus allows to extend easily the Schochet argument concerning these aspects. On the contrary, a generic, additive or multiplicative, noise would fail at this point, though uniform bounds could work also for certain choices of noise.

Secondly, the Schochet argument does not immediately goes through with respect  the uniform $L^2(\mt^2)$ bounds on the velocity $u$. Indeed,  as we have seen above,  the equation  \eqref{eq:stochEulervel_intro} for the velocity filed $u$ does not preserve the energy. To overcome this difficulty, we assume a certain structure of
the coefficients $\sigma_k$ so that  uniform $L^2(\mt^2)$ bounds on $u$ can be still deduced.
These assumptions are not restrictive as they allow  to deal with the relevant cases, as for example the "locally isotropic" convariance matrices, see the discussion in Section \ref{sec:hp_noise}.

The  last comment we make in this Introduction concerns  an alternative approach  to  the proof. In \cite[Section 7]{BrzFlaMau2016}, we suggested an alternative proof of the well-posedness within the class of  bounded solutions, which is based on the Doss-Sussmann transformation.  To be precise,  $\psi$ is the stochastic flow solutions corresponding to the following SDEs on the torus $\mathbb{T}^2$:
\begin{align*}
d\psi_t = \sum_k\sigma_k(\psi_t) \circ dW^k_t,
\end{align*}
then the process  $\tilde{\xi}(t,x,\omega):= \xi(t,\psi(t,x,\omega),\omega)$ satisfies a random Euler-type PDE of the following type
\begin{align}
\begin{aligned}\label{eq:random_PDE}
&\partial_t \tilde{\xi} +\tilde{u} \cdot\nabla \tilde{\xi} =0,\\
&\tilde{u}(t,x,\omega) = D\psi^{-1}(t,x,\omega) \int_{\mt^2} K(\psi(t,x,\omega)-\psi(t,y,\omega)) \tilde{\xi}(t,y,\omega)dy,
\end{aligned}
\end{align}
where $\psi^{-1}$ is the inverse flow. Equation \eqref{eq:random_PDE} is a nonlinear transport type equation, where the kernel $K$ has been replaced by the above random kernel. In the case that $\sigma_k=e_k 1_{k\in\{1,2\}}$, where $e_k$ is the canonical basis on $\mr^2$, the PDE \eqref{eq:random_PDE} is exactly the Euler vorticity equation. Indeed in this case the stochastic Euler equations correspond simply to a random shift in Lagrangian coordinates. In the general case however, this is not the Euler vorticity equation, yet it enjoys similar properties since the random kernel has similar regularity properties, hence the well-posedness among bounded solutions can be derived applying the deterministic arguments to the random PDE \eqref{eq:random_PDE}. One may then wonder if a similar argument is possible here, namely using Schochet's arguments for the random PDE \eqref{eq:random_PDE}. In this context however, there are two problems at least. Firstly, even if we could just apply the Delort's and/or Schochet's result to equationm \eqref{eq:random_PDE}, at $\omega$ fixed, this would only give the existence of a measure-valued and $H^{-1}$-valued solution to \eqref{eq:random_PDE} at $\omega$ fixed, with no adaptedness property; so a compactness argument would probably be needed anyway. Secondly, from \eqref{eq:random_PDE} it is not immediately clear how to get a uniform $H^{-1}(\mt^2)$ bound on $\tilde{\xi}$. Moreover this type of arguments \`a la Doss-Sussmann does not work intrinsically on the stochastic PDE. Let us remark anyway that the argument may still give some hints: for example, it would be interesting to see if, using the equation for $\tilde{u}$ instead of the equation for $u$, one can get a uniform $H^{-1}(\mt^2)$ bound without the additional assumptions on $\sigma$ that are needed here. We leave this point for future research.

\section{The setting}

\subsection{Notation}

We recall some notation frequently used in the paper. We use the letters $t$, $x$, $\omega$ for a generic element in $[0,T]$, $\mt^2$, $\Omega$ respectively. The coordinates of $x$ are denoted by  $(x^1,x^2)$, while the partial derivatives are denoted by $\partial_{x_j}=D_j=\frac{\partial}{\partial x_j}$, $j=1,2$. Unless differently specified, the derivatives $\nabla$, $D$ and  $\Delta$ are meant to be  with respect to the space $x$. The Sobolev, not the Besov,  spaces are denoted by $\nH^{s,p}$, for $s \in \mathbb{R}$ and $p\in [1,\infty)$.  For $p=2$, we put $H^s = \nH^{s,2}$.

By $L_0^2(\mt^2,\mr^2)$ we denote the closed Hilbert subspace of $L^2(\mt^2,\mr^2)$ consisting of of  mean zero vector fields
and  by $\rH$ we denote the closed Hilbert subspace of $L_0^2(\mt^2,\mr^2)$ consisting of  divergence free  vector fields. Alternatively,
$\rH$  is equal to the closed Hilbert subspace of $L^2(\mt^2,\mr^2)$ consisting of  mean zero and divergence free  vector fields. We endow $\rH$ with the inner product from the space
$L^2(\mt^2,\mr^2)$ and introduce the Leray-Helmholtz projection $\Pi$ which by definition is the orthogonal projection from $L^2(\mt^2,\mr^2)$ onto $\rH$.

For a differentiable map $f=(f^1,\cdots,f^m):\mt^2\to \mr^m$, $m\in \mathbb{N}$,
we denote by $Df(a)$ the $m \times 2 $-matrix of the Fr{\`e}chet derivative $d_af \in \mathcal{L}(\mathbb{R}^2,\mathbb{R}^m)$ of $f$ at $a=(a_1,a_2)\in \mt^2$,  in the standard bases of $\mathbb{R}^2$ and $\mathbb{R}^m$,  i.e.  $Df(a) =(D_{j}f^i)=\begin{pmatrix} D_1 f^1 & D_2 f^1 \\
\ldots & \ldots\\
D_1 f^m & D_2 f^m
\end{pmatrix}$.  If $b: \mt^2 \to \mr^2$ then we put \[b \cdot \nabla f:=\bigl\{ \mt^2 \ni x \mapsto  \sum_{j=1}^2b_j(x) D_jf(x) \in \mr^m\}.
\]
If $m=1$, then by $\nabla f(a)$ we will understand the unique vector in $\mathbb{R}^2$ such that $\lb \nabla f(a),x\rb=d_af(x)$, $x\in \mathbb{R}^2$.
 By $a \otimes b$ we denote the tensor product of two vectors  $a,b\in \mathbb{R}^2$, i.e. a linear operator in $\mathbb{R}^2$ defined by $\mr^2 \ni x \mapsto a \lb b,x\rb \in \mr^2$. The matrix of
 in the standard bases of $\mr^2$ is $[a_ib_j]_{i,j=1}^2=\mat{a_1b_1}{a_1b_2}{a_2b_1}{a_2b_2}$.
 Identifying $a$ with a $2\times1$-matrix, we see that the last  matrix  is equal to $ab^{\mathrm{t}}$.

For function  spaces, we often put the input variables ($t$, $x$, $\omega$) as subscripts. For example, the spaces $L^2(\mt^2)$, $C([0,T])$, is denoted in short by $L^2_x$, $C_t$ respectively.  As another example, the space $C([0,T];(H^{-4}(\mt^2),w))$ of continuous functions on $[0,T]$, which take values in $H^{-4}(\mt^2)$ endowed with  the weak topology, is denoted in short by $C_t(H^{-4}_x,w)$. The symbol $f\ast g$ stands for the convolution in the space variable  using the group structure of  $\mt^2$,  between two functions or distributions $f$ and $g$ on $\mt^2$.

The space $\mc{M}_x=\mc{M}(\mt^2)$ is the space of finite $\mathbb{R}$-valued  \Radon measures on $\mt^2$. It is a Banach space, endowed with the total variation norm $\|\cdot\|_{\mc{M}_x}$ and is the dual of the space $C_x= C(\mt^2)$ of the space of continuous functions on $\mt^2$. The notation $\lan\mu,\varphi\ran$ will be used to denote the duality pairing  between a measure $\mu \in \mc{M}_x$ and a function $\varphi \in C_x$. The closed ball of center $0$ and radius $M$ on $\mc{M}_x$ is denoted by $\mc{M}_{x,M}$; the radius here refers to the  norm $\|\cdot\|_{\mc{M}_x}$. Following the above notation, the space $C([0,T];(\mc{M}_{x,M},w^\ast ))$ of continuous functions on $[0,T]$ taking values in the closed ball of radius $M$ in  $\mc{M}_x$ endowed with the weak-$^\ast$ topology, is denoted in short by  $C_t(\mc{M}_{x,M},w^\ast )$. We  also use notation $\mc{M}_{x,+}$ to denote  the set of non-negative finite \Radon measures on $\mt^2$. Finally, by  $\mc{M}_{x,M,+,\mathrm{no-atom}}$ we denote  the set of non-negative non-atomic \Radon measures with total mass $\le M$, $\mc{M}_{x,y}$ for the space of finite \Radon measures on $\mathbb{T}^{2\times 2}$.

Given a probability space $(\Omega,\mc{A},\mathbb{P})$ endowed with a filtration $\mathbb{F}=(\mc{F}_t)_t$, by the  symbol $\mathcal{B}\mathbb{F}$ we  denote  the progressively measurable $\sigma$-algebra associated with $\mathbb{F}$. The letter $C$ will be used for constants which may change from one line to another.


\subsection{Assumptions on the noise}\label{sec:hp_noise}


Here we give the assumptions on the noise coefficients $\sigma_k$.

\begin{condition}\label{assumption_sigma}
We assume the following assumptions.
\begin{itemize}
\item[\textbf{(A.1)}] For every $k\in \mathbb{N}$,  $\sigma_k \in C^1(\mt^2,\mr^2)\cap \rH$,  and
\begin{align*}
\sum_k\|\sigma_k\|_{C^1_x}^2<\infty.
\end{align*}
In particular, we assume that there exists a continuous function $a:\mt^2\times\mt^2\to \mathcal{L}(\mr^2,\mr^2)
$, called the infinitesimal covariance matrix, such that

\begin{align}
\label{eqn-a}
a(x,y) =\sum_k 
\sigma_k(x)\otimes \sigma_k(y)
, \;\;\; (x,y)\in \mt^2,
\end{align}
with the convergence in the space $C_{x,y}$. Moreover, function  $a$ is differentiable separately w.t.t.  $x$ and  $y$ and
\begin{align*}
\partial_{x_i}a(x,y) =\sum_k \partial_{x_i}\sigma_k(x) \otimes \sigma_k(y) ,\quad i=1,2,
\end{align*}
with the convergence in $C_{x,y}$. Analogously for the derivatives  w.r.t. $y$.
\item[\textbf{(A.2)}] There exists $c\ge 0$ such that the function $a$ defined above satisfies,
\begin{align}\label{eqn-axx}
a(x,x)=cI_2,\;\; x\in \mt,
\end{align}
where $I_2$ is the identity map in $\mr^2$. 
\item[\textbf{(A.3)}] The function $a$ defined above satisfies
\begin{align}\label{eqn-axy}
\partial_{y_i}a(x,y)\mid_{y=x} =0,\quad i=1,2,\quad  x\in\mt^2.
\end{align}
\end{itemize}
\end{condition}

\begin{remark}
\label{rem-second assumtion}
The equality \eqref{eqn-axx} in Assumption  \textbf{(A.2)} reads as follows:
\begin{align*}
&\sum_k\sigma_k^i(x)\sigma_k^j(x) = c\delta_{ij},\quad  x,\quad  i,j=1,2.
\end{align*}
The equality \eqref{eqn-axy} in Assumption  \textbf{(A.3)} reads as follows:
\begin{align}\label{eqn-axy'}
&\sum_k\sigma_k^i(x) \partial_{x_l} \sigma_k^j(x) = \partial_{y_l}a^{i,j}(x,y)\mid_{y=x} = 0,\quad  x,\quad  i,j,l=1,2.
\end{align}

Assumption \textbf{(A.2)} allows us  to simplify the It\^o-Stratonovich correction term  \[\frac12\sum_k \sigma_k\cdot\nabla[\sigma_k\cdot\nabla \xi]\]  in equation \eqref{eqn: Euler stoch vorticity form-Intro}. It cancels the first order term, i.e.
\[\frac12\sum_k (\sigma_k\cdot\nabla) \sigma_k \cdot\nabla \xi=
\frac 12 \sum_k \sum_{i,j} \sigma_k^i D_i \sigma_k^jD_j \xi=0
 \]  and  makes the second order term equal to a constant multiplied by the Laplacian, i.e.
\[\frac12 \sum_k \sum_{i,j} \sigma_k^i\sigma_k^jD_iD_j \xi= \frac12 \sum_k \tr [\sigma_k\otimes \sigma_k
D^2\xi]=\frac{c}{2}\Delta \xi.\] The latter equality is obvious and   former   follows from equality \eqref{eqn-axy'} with $l=i$. Indeed,
\begin{align*}
& \sum_k \sum_{i,j} \sigma_k^i D_i \sigma_k^jD_j \xi=\sum_{i,j} \bigl[   \sum_k  \sigma_k^i D_i \sigma_k^j\big] D_j \xi = \sum_{i,j} 0 \times  D_j\xi=0.
\end{align*}
Hence the Stratonovich integral  in equation \eqref{eqn: Euler stoch vorticity form-Intro} takes the following form
\begin{align}\label{eqn-Stratonovich integral}
\int_0^t \sum_k \sigma_k\cdot \nabla \xi \circ d{W}^k(s) =\sum_k \sigma_k\cdot \nabla \xi d{W}^k  - \frac{c}2 \int^t_0 \Delta \xi(x) \, dr.
\end{align}
Finally, we infer that  the vorticity equation \eqref{eqn: Euler stoch vorticity form-Intro}  can be written as the following  It\^o equation:
\begin{align}
\begin{aligned}\label{eq:stochEuler_intro_Ito}
&\partial_t \xi + u\cdot \nabla \xi + \sum_k \sigma_k\cdot \nabla \xi \,\,\dot{W}^k = \frac12 c\Delta\xi,\\
&u(t) = K \ast \xi(t), \;\;\; t \geq 0.
\end{aligned}
\end{align}

The assumption \ref{assumption_sigma} for the for well-posedness of bounded vorticity solutions  in  the present paper are   stronger than those  used  in \cite{BrzFlaMau2016}.
Indeed, in \cite{BrzFlaMau2016} only the first two assumptions are needed;  the second in a slightly weaker form. Moreover,  the second assumption is not essential  because the  first-order It\^o correction can be treated via lengthly computations. On the other hand, in the current paper,   it seems unclear from the proof whether Assumptions \textbf{(A.2)} and \textbf{(A.3)}  can be removed. Indeed these two assumptions guarantee that the equation for the velocity $u$ has no other  first order terms but   the transport one. The last property   in turn implies a key energy bound of solution, i.e.  the $L^2_x$ norm of the solution $u(t)$.
\end{remark}

The following Example provides a non-trivial example, i.e.  non constant, of vector fields   $\sigma_k$ satisfying Assumption \ref{assumption_sigma}.

\begin{example}
Let $\beta>3$ and define, with $\mathbb{Z}_{\ast}^2 :=\{ k\in\mathbb{Z}^2 \} \setminus \{0\} $ and,  $k^\perp=(-k_2,k_1)$ for $k=(k_1,k_2)\in\mathbb{Z}^2$,
\begin{align*}
\sigma_k(x) = (\cos(k\cdot x)+\sin(k\cdot x)) \frac{k^\perp}{|k|^\beta},\;\;\; x\in \mt^2, \;\; k \in\mathbb{Z}_{\ast}^2\setminus \{0\}.
\end{align*}
Since $\beta>3$, we infer that $\sum_k\|\sigma_k\|_{C^1_x}^2 <\infty$ and we can calculate that the function $a$ defined in \eqref{eqn-a} satisfies
\begin{align*}
a(x,y) &= \sum_{k\in \mathbb{Z}_{+}^2} \cos(k\cdot (x-y)) \frac{k^\perp \otimes k^\perp}
{|k|^{2\beta}},
\end{align*}
and \eqref{eqn-axx} with $c=\sum_{k\in\mathbb{Z}_{+}^2} \frac{2}{|k|^{2\beta-2}}$,
 where
\[
\mathbb{Z}_{+}^2:=\{ k\in\mathbb{Z}^2: k_1\ge0, k_2>0 \}.
\]
\end{example}

Note that, in this example, the infinitesimal covariance matrix $a$ is translation-invariant, i.e  $a(x,y) = a(x-y)$,  and even, i.e $a(x)=a(-x)$. Moreover, if the function $a$ is translation invariant and even, then it satisfies  Assumption \textbf{(A.3)}. 
Indeed
\begin{align*}
\partial_{y_i}a(x,y)\mid_{y=x} = -\partial_{z_i}a(z)\mid_{z=0} =0.
\end{align*}

We could \textit{essentially} include the class of isotropic infinitesimal covariance matrices in our setting. On $\mr^2$, an infinitesimal covariance matrix $a$ is called isotropic if it is translation- and rotation-invariant, that is $a(x,y)=a(x-y)$ for all $(x,y)$ and $R^{\mathrm{t}}a(Rx)R =a(x)$ for every rotation matrix $R$, and $a(0)=cI$ for some $c>0$. Rotation invariance implies that $a$ is even, take for instance $Rx=-x$. Hence a sufficiently regular isotropic matrix $a$ satisfies the Assumption \ref{assumption_sigma}. More precisely, if  $a$ is regular, then  one can find vector fields $\sigma_k$  safisfying Assumption \ref{assumption_sigma}.

The isotropy condition on $a$ means essentially that the noise $\sum_k \sigma_k(x)\circ \dot{W}_t$ is Gaussian, white in time, coloured and isotropic in space. Such a class has been considered in the mathematical and physical literature, see e.g. \cite{BaxHar1986}, \cite{Gaw2008}, even without the nonlinear term: for example, the same type of noise, but irregular in space, provides a simplified model for the study of passive scalars in a turbulent motion (see \cite{Gaw2008}, \cite{FalGawVer2001}, \cite{LeJRai2002}). Strictly speaking we are not allowed here to take an isotropic matrix as infinitesimal covariance matrix $a$ here, for the simple reason that the torus itself (considered as $[-1,1]^2$ with periodic boundary conditions) is not rotation-invariant. However, one may still take $a$ translation-invariant and rotation-invariant on a neighborhood of the diagonal $\{x=y\}$. Moreover the torus setting here is taken to avoid technicalities at infinity, but we believe a similar construction, including isotropic vector fields, would go through also in the full space case.

%

\subsection{The nonlinear term}

We focus now on the definition and properties of the nonlinear term $\lan \xi,u^\xi\cdot\nabla \varphi\ran$. For this, we introduce some important notation. By  $\mc{M}_x=\mc{M}(\mt^2)$ we denote  the space of $\mathbb{R}$-valued \Radon measures on $\mt^2$. It is a Banach space, endowed with the total variation norm $\|\cdot\|_{\mc{M}_x}$.  The space  $\mc{M}_x$ is the dual of the space $C_x= C(\mt^2)$ of  all continuous $\mathbb{R}$-valued functions on $\mt^2$,  see e.g. \cite[Theorem 7.4.1]{Dudley_2002}.  The space $\mc{M}_x$ is endowed with the Borel $\sigma$-field generated by the \textit{weak-$^\ast$} topology. The notation $\lan f, g \ran$ denotes either the $L^2$ product between two functions
in $L^2$ or a duality between a  distribution and a test function  on $\mt^2$.  Similarly,  $\lan\mu,\varphi\ran$ will be also used to denote the duality  between a measure $\mu \in \mc{M}_x$ and a function $\varphi \in C_x$.

It follows from  Lemma \ref{lem:Green_function} that if  $\xi \in L^{4/3}_x$, then $u^\xi \in \nH^{1,4/3}_x$ and thus  $u^\xi \in \ L^4_x$ by the Sobolev Embedding Theorem, see \cite[Theorem 6.5.1]{Bergh+Lofstrom_1976}.  Hence the product $\xi u^\xi \in L^1_x$ and the nonlinear term makes sense for a smooth function  $\varphi$. However, if $\xi$ does not belong  to  $L^{4/3}_x$, the product $\xi u^\xi$ is, in general,  not defined.
To overcome this difficulty, we use the following Lemma, due to Schochet \cite[Lemma 3.2 and discussion thereafter]{Sch1995}. For $\varphi \in C^2_x$, we define the function $F_\varphi:\mt^2\times\mt^2\to\mr$ by
\begin{align}
\label{eqn-F_varphi}
F_\varphi(x,y) := \frac12 K(x-y) \cdot (\nabla_x\varphi(x)-\nabla_y\varphi(y))1_{x\neq y},\;\; (x,y)\in \mt^2\times\mt^2.
\end{align}

\begin{lemma}\label{lem:Poupaud_trick}
The following properties are satisfied.
\begin{itemize}
\item If $\varphi \in C^2_x$, the the function $F_\varphi$ defined in \eqref{eqn-F_varphi} is bounded  by $C\|\varphi\|_{C^2_x}$ and continuous outside the diagonal $\{(x,y): x=y\}$. Moreover, the averages of $F_\varphi$ with respect to both $x$ and $y$ are equal to zero, i.e.
\begin{align}
\label{eqn-F average}
\int F_\varphi(x,y)  \rmd x =0 \,\,\forall y\in ,\quad \int F_\varphi(x,y)  \rmd y =0\,\,\forall x\in \mt^2.
\end{align}
\item For any measure $\xi \in \mc{M}_x$, the formula
\begin{align*}
C^{\infty}_x \ni \varphi \mapsto \lan N(\xi), \varphi \ran := \int_{\mt^2}\int_{\mt^2} F_\varphi(x,y) \xi(\rmd x)\xi(\rmd y) \in \mathbb{R},
\end{align*}
has a unique extension to a bounded liner functional on $H^4_x$, i.e.
$N(\xi) \in H^{-4}_x$> Moreover, there exists $C>0$ such that
\begin{align}\label{ineq-N}
\|N(\xi)\|_{H^{-4}_x} \le C\|\xi\|_{\mc{M}_x}^2, \;\;  \xi \in\mc{M}_x.
\end{align}
\item The map $N:\mc{M}_x\to H^{-4}_x$ defined above is Borel measurable, where  $\mc{M}_x$ is endowed with the Borel $\sigma$-algebra generated by the weak-$^\ast$ topology and
$H^{-4}_x=(H_x^4)^\ast$ is endowed with the Borel $\sigma$-field generated by the weak-$^\ast$ topology as well.
\item The map $N$ coincides with the nonlinear term in equation \eqref{eqn: Euler stoch vorticity form-Intro} on  $ L^{4/3}_x$, i.e.
\begin{align}
N(\xi) =\lan \xi,u^\xi\cdot\nabla \varphi\ran, \;\; \xi \in L^{4/3}_x,\; \varphi \in C^{\infty}_x. \label{eq:nonlin_poupaud}
\end{align}
\item For every $M>0$, the map $N$, restricted to the set $\mc{M}_{x,M,+,\mathrm{no-atom}}$ of non-negative non-atomic measures on $\mt^2$ with total mass $\le M$ and with values in $H^{-4}_x$, is continuous, i.e.  $\xi\mapsto \lan N(\xi),\varphi\ran$ is continuous for every $\varphi \in H^4$.  Moreover,  $N$ is the unique  continuous extension of the nonlinear term $L^{4/3}_x \ni \varphi \mapsto \lan \xi,u^\xi\cdot\nabla \cdot \ran \in H^{-4}_x$    to  $\mc{M}_{x,M,+,\mathrm{no-atom}}$.
\item If $\xi \in \mc{M}_x$ and $\alpha \in \mathbb{R}$, then $N(\xi+\alpha)=N(\xi)$.
\end{itemize}
\end{lemma}

Thanks to this Lemma, we will use the right-hand side of identity \eqref{eq:nonlin_poupaud} as the definition of the nonlinear term $\xi \nabla u^\xi$   in equation \eqref{eqn: Euler stoch vorticity form-Intro}. Note that  by Lemma \ref{lem:no_atoms} stated later, the set $\mc{M}_{x,M,+,\mathrm{no-atom}}$ contains the set  $\xi \in \mc{M}_{x,+}\cap H^{-1}_x$ we are interested in here. The last assertion of Lemma \ref{lem:Poupaud_trick} is a consequence of Lemma \ref{lem:continuity_mass}, while the other assertions are proved in  Appendix \ref{app-nonlinear term}.

\subsection{A definition of a measure-valued solution}

Now we can give the definition of a measure-valued solution to the stochastic vorticity equation. Again we give some notation. For any fixed $M>0$, $\mc{M}_{x,M}$ is the set of all finite $\mathbb{R}$-valued  \Radon measures $\mu$ on $\mt^2$ with total variation $\|\mu\|_{\mc{M}_x}\le M$. This will be preceded by presenting some additional notation.  Let us first observe that by the Banach-Alaoglu Theorem and  \cite[Theorem 3.28]{Brezis-2011}, the set  $\mc{M}_{x,M}$ endowed with the weak-$^\ast$ topology induced from  $\mc{M}_x$,  is a compact Polish space. This set is endowed  with the Borel $\sigma$-field  generated by this topology. Given a filtration $\mathbb{F}=(\mc{F}_t)_t$, we denote by  $\mathcal{B}\mathbb{F}$ the associated progressive $\sigma$-field.  By an $\mathbb{F}$-cylindrical Wiener process  we understand  a sequence $W=(W^k)_k$ of independent $\mathbb{F}$-Brownian motions.

The following  definition provides a rigorous   formulation of equation \eqref{eqn: Euler stoch vorticity form-Intro} in the It\^o form \eqref{eq:stochEuler_intro_Ito}.

\begin{definition}\label{def:sol}
Assume that  $M>0$ and $T>0$. Assume that the vector fields $\sigma_k$, $k\in \mathbb{N}$,  satisfy  Assumption \ref{assumption_sigma}. A weak distributional $\mc{M}_{x,M}$-valued solution to the vorticity equation \eqref{eqn: Euler stoch vorticity form-Intro} is an object \[(\Omega,\mc{A},\mathbb{F},\mathbb{P},W,\xi),\]
 where $\mathbb{F}=(\mc{F}_t)_{t\in [0,T]}$ is a filtration,
$(\Omega,\mc{A},\mathbb{F},\mathbb{P})$ is a filtered probability space satisfying the usual assumptions,  $W=(W^k)_k$ is an $\mathbb{F}$-cylindrical Wiener processon time interval $[0,T]$,
$\xi:[0,T]\times\Omega\to\mc{M}_{x,M}$ is $\mathcal{B}\mathbb{F}$ Borel measurable function and, $\mathbb{P}$-a.s.\footnote{The $\mathbb{P}$-exceptional set being independent of $t$.} the following equalities are satisfied
  in $H^{-4}_x$,
\begin{equation}
\label{eq:stochEulervort}
\xi_t = \xi_0 -\int^t_0 N(\xi_r) \rmd r -\sum_k \int^t_0 \sigma_k\cdot\nabla \xi_r \rmd W^k_r +\frac{c}2 \int^t_0 \Delta \xi_r \rmd r,\quad  t\geq 0.
\end{equation}
\end{definition}
\begin{remark}\label{rem-T}
The above definition as well as Theorem \ref{thm:main} are formulated on an arbitrary (but fixed) time interval $[0,T]$. One can change these to an infinite interval $[0,\infty)$, see for instance \cite{Brz+Ferr_2019}.
\end{remark}

Let us observe that by Lemmata \ref{lem:Poupaud_trick} and \ref{lem:H_Borel}  the map $\xi :[0,T]\times\Omega \to H_x^{-4}$  as well as  the integrands in \eqref{eq:stochEulervort} are $\mc{B}\mathbb{F}$ Borel measurable $H^{-4}_x$-valued maps.  Moreover, by inequality \eqref{ineq-N} we infer that
\begin{align}
\E \sum_k \int^T_0 \|\sigma_k\cdot\nabla \xi_r\|^2_{H^{-4}_x} \rmd r \le M^2 \sum_k\|\sigma_k\|_{C_x}^2.\label{eq:welldef_stoch_int}
\end{align}
Hence the deterministic Bochner integrals and the stochastic It\^o integral exists in the Hilbert space  $H^{-4}_x$.

Moreover, Lemma \ref{lem:stochEulervort_H} shows that \eqref{eq:stochEulervort} is equivalent to the following weak formulation, i.e. that  for every $\varphi \in C^\infty_x$,
\begin{align}
\begin{aligned}\label{eq:stochEulervort_test}
\lan \xi_t,\varphi \ran &= \lan \xi_0,\varphi\ran +\int^t_0 \lan N(\xi_r),\varphi \ran \rmd r +\sum_k \int^t_0 \lan \xi_r,\sigma_k\cdot\nabla \varphi \ran \rmd W^k_r \\
&\ \ \ +\frac12 \int^t_0 \lan \xi_r, c\Delta \varphi \ran \rmd r,\quad\text{for every }t,\quad \mathbb{P}-\text{a.s.},
\end{aligned}
\end{align}
with the $\mathbb{P}$-exceptional set being independent of $t$ and test functions $\varphi$.

We will sometimes say that a process $\xi$ is an $L^p_x$-valued solution of problem \eqref{eqn: Euler stoch vorticity form-Intro} iff it  has finite $L^m_{t,\omega}(L^p_x)$ norm for some $1\le m\le \infty$, where we identify a measure with its density if the latter  exists. Similarly we can define am   $H^{-1}_x$-valued solutions.

\section{The global existence for initial vorticity in the space  $\mc{M}_{x,+}\cap H^{-1}_x$}\label{sec:main_strategy}

The objective  of this section is to present the main result of this paper.

\begin{theorem}\label{thm:main}
Assume the Assumption \ref{assumption_sigma} on the coefficients $\sigma_k$.  Then for all    $M>0$ and  $T>0$,  and every  $\xi_0 \in \mc{M}_{x,M,+}\cap H^{-1}_x$, there exists a weak distributional $\mc{M}_{x,M}$-valued solution $(\Omega,\mc{A},\mathbb{F},\mathbb{P},(W^k)_k,\xi)$, where $\mathbb{F}=(\mc{F}_t)_{t\in [0,T]}$,  to the vorticity equation \eqref{eqn: Euler stoch vorticity form-Intro} such that   $\xi \in C_t(\mathcal{M}_{x,M,+},w^\ast )\cap L^2_{t,\omega}(H^{-1}_x)$ $\mathbb{P}$-a.s..
\end{theorem}

\begin{remark}
Note that, if $\xi$ is a solution to the vorticity equation \eqref{eqn: Euler stoch vorticity form-Intro} and $\alpha \in \mathbb{R} $ is a real constant, 
then $\xi+\alpha$ is also a solution, see Lemma \ref{lem:Poupaud_trick}. Hence the Theorem \ref{thm:main} can be generalized as follows. \\
 For every $\xi_0 \in \mc{M}_{x,M}\cap H^{-1}_x$ with negative part bounded by a constant $\alpha$, then there exists a weak distributional $\mc{M}_{x,M}$-valued solution $\xi$ to \eqref{eqn: Euler stoch vorticity form-Intro},  whose paths belong to  $C_t(\mathcal{M}_{x,M,+},w^\ast )\cap L^2_{t,\omega}(H^{-1}_x)$, with negative part bounded by $\alpha$,  for all $t$, $\mathbb{P}$-a.s..

This generalization is important because,  in particular,   if our departure point is the initial   velocity $u_0$, then, unless $\curl  (u_0) =0$,  $\curl  (u_0)$ cannot be non-negative, but it can be bounded from below.
\end{remark}

The strategy of the proof of Theorem \ref{thm:main} is  as follows. We begin by fixing $M>0$ and $T>0$.

\textbf{Compactness argument}. We take  a family $\xi^\epsilon$, $\eps>0$,  of solutions with regular bounded initial datum $\xi^\epsilon_0$. We will show that the family of laws of these solutions is  tight on a suitable
topological space,  via appropriate  a'priori bounds. This  will be accomplished in the following  steps.
\begin{enumerate}
\item Prove the non-negativity and a uniform $L^\infty_{t,\omega}(\mathcal{M}_{x,+})$ bound on $\xi^\eps$.  This will follow from the conservation of non-negativity and from the conservation of mass for the vorticity equation \eqref{eqn: Euler stoch vorticity form-Intro}.
\item Prove a uniform $L^2_{t,\omega}(H^{-1}_x)$ bound on $\xi^\eps$.  Since the $L^2_{t,\omega}(H^{-1}_x)$ norm of $\xi$ is equivalent to the $L^2_{t,\omega}(L^2_x)$ norm of the corresponding velocity $u^\eps:=u^{\xi^\eps}$,  we will  prove a uniform energy bound on $u^\eps$   by rewriting the vorticity equation \eqref{eqn: Euler stoch vorticity form-Intro}  in an equivalent form \eqref{eqn: Euler stoch velocity form} for the velocity $u^\eps$. Note that, contrary  to the deterministic case, the energy, i.e. the $L^2$ norm of the velocity $u^\eps$,  is not preserved. This is  due to the presence of an additional term $(\nabla \sigma_k)\cdot u \circ \,\dot{W}^k $ in the equation. To prove the energy bound, the assumptions on $\sigma_k$ play a crucial role.
\item Prove, for $\alpha<1/2$,  a uniform $L^2_\omega(C^\alpha_t(H^{-4}_x))$ bound on $\xi^\eps$. This will  follow from the Lipschitz bounds w.r.t the  time variable  of the Bochner  integrals in the vorticity equation \eqref{eqn: Euler stoch vorticity form-Intro} and the H\"older bound w.r.t the  time variable of  the $H^{-4}$-valued stochastic integrals.
\item Show tightness  of laws of $u^\eps$ on  $C_t(\mathcal{M}_{x,M},w^\ast )\cap (L^2_t(H^{-1}_x),w)$, where by $w^\ast$, respectively $w$,   we denote  the weak-$^\ast$ topology on $\mathcal{M}_{x,M}$, respectively  the weak topology on $L^2_t(H^{-1}_x)$).  This will follow from the previous uniform bounds.  In fact, we could   prove tightness simply in $C_t(\mathcal{M}_{x,M},w^\ast )$, without using the $L^2_t(H^{-1}_x)$ bound, but  our present proves to be useful in the convergence part.
\end{enumerate}

\textbf{Convergence argument}. We will show that any limit point $\xi$ of $\xi^\epsilon$, as $\eps \to 0$,  solves the vorticity equation \eqref{eqn: Euler stoch vorticity form-Intro}. This  will be accomplished in the following  steps.
\begin{enumerate}
\item Pass to an a.s. convergence.  By the Skorokhod-Jakubowski Theorem, see \cite{Jak1997} and \cite{Brz+Ondr_2011}, we have, up to a  subsequence and a choice of a new probability space, an a.s. convergence of  $\xi^\eps$ to some $\xi $ with trajectories  in $ C_t(\mathcal{M}_{x,M},w^\ast )\cap (L^2_t(H^{-1}_x),w)$.
\item Show that   $\xi$  satisfies \eqref{eq:stochEulervort}.  For the nonlinear term, we will use the Schochet approach and in particular  the continuity of the nonlinear term among non-negative non-atomic measures as in Lemma \ref{lem:continuity_mass} and the fact that $H^{-1}_x$ measures are non-atomic. For the stochastic term, we use the approximation of the stochastic integral via the Riemann sums as in \cite{BrzGolJeg2013}.
\end{enumerate}

The main result will follow from Lemma \ref{lem:limit_sol}, which will show that the limit $\xi$ of any subsequence satisfies \eqref{eq:stochEulervort}.

\section{Proof of the main Theorem \ref{thm:main} }

We fix $M>0$ and $T>0$ and the initial condition $\xi_0 \in \mc{M}_{x,M,+}\cap H^{-1}_x$. We also assume, for the whole section,  that the coefficients $\sigma_k$ satisfy
 Assumption \ref{assumption_sigma}.

\subsection{A priori bounds}\label{sec:apriori_bd}

For $\eps>0$, we take $\xi^\eps$ to be the $L^\infty_{t,x,\omega}$ solution to the stochastic vorticity equation \eqref{eq:stochEulervort}, with initial condition $\xi_0^\eps = \xi_0*\rho_\eps$, where $(\rho_\eps)_{\eps >0}$ is the standard approximation of identity  on $\mt^2$. The convolution is understood w.r.t. the group structure of $\mt^2$. Note that by inequality (4) in \cite[section 9.14]{Rudin-FA_1991},   $ \|\xi_0^\eps\|_{\mc{M}_x} \leq M$.

The existence and the uniqueness of such  solution follows from our joint paper  \cite{BrzFlaMau2016} with Flandoli. To be more precise,    \cite[Theorem 2.1 and Definition 2.12]{BrzFlaMau2016}  implies the existence and the  uniqueness of the stochastic flow $\Phi^\eps: [0,\infty)\times \mt^2\times \Omega \to \mt^2 $
such that
\begin{trivlist}
\item[(i)]
for almost every $\omega \in \Omega$, the map $[0,\infty)\times \mt^2 \ni (t,x) \mapsto \Phi^\eps(t,x,\omega) \in  \mt^2 $ is continuous,
\item[(ii)]
for every $x \in \mt^2$, the process
$\Phi^\eps(\cdot,\cdot,\omega): [0,\infty)\times \mt^2  \to \mt^2 $ is progressively measurable,
\item[(iii)]
for almost every $\omega \in \Omega$,  $\Phi^\eps(t,\cdot,\omega):  \mt^2  \to \mt^2 $ preserves the Lebuesgue
measure for every $t\geq 0$,
\item[(iv)]
and for every $x\in \mt^2$, the process $\Phi^\eps(t,x)$ is solution to the following  SDE
\begin{align*}
\rmd \Phi^\eps(t,x) = \int_{\mt^2} K(\Phi^\eps(t,x)-\Phi^\eps(t,y)) \xi^\eps_0(y) \rmd y \rmd t +\sum_k \sigma_k(\Phi^\eps(t,x)) \rmd W^k_t.
\end{align*}
\end{trivlist}

Moreover,  if we define  $\xi^\eps_t(\omega)$ to be  the image (i.e. the push-forward) of  measure $\xi_0^\eps$, see \cite{Dudley_2002},  by the map $\Phi(t,\cdot,\omega)$), i.e.
\begin{align}
\xi^\eps_t(\omega)=(\Phi(t,\cdot,\omega))_\# \xi^\eps_0, \;\; t\geq 0,
\label{eq:repr_formula}
\end{align}
 then,  by  \cite[Proposition 5.1]{BrzFlaMau2016} we infer that  for a.e.~$\omega$, for every $t \geq 0$, the random variable  $\xi^\eps_t(\omega)$ admits a density with respect to the Lebesgue measure, which belongs to  $L^\infty_{t,x,\omega}$ and, for every $\varphi \in C^\infty_x$, the process $\lan \xi^\eps_t,\varphi\ran$ is progressively measurable and satisfies \eqref{eq:stochEulervort_test}. It follows, see Remark \ref{rmk:Linfty_sol}, that, up to taking an indistinguishable version, such $\xi^\eps$ is a weak distributional $\mc{M}_{x,M}$-valued solution in the sense of Definition \ref{def:sol}.

\subsubsection{Bound in $L^\infty_{t,\omega}(\mc{M}_x)$}

We start with a uniform $L^\infty_{t,\omega}(\mc{M}_x)$ bound:

\begin{lemma}\label{lem:cons_mass}
For every $\eps>0$ fixed, for a.e.~$\omega$, we have
\begin{align*}
\sup_{t\in[0,T]}\|\xi^\eps_t\|_{L^\infty_x} \leq \sup_{t\in[0,T]}\|\xi^\eps_t\|_{\mc{M}_x}  \leq  \|\xi^\eps_0\|_{\mc{M}_x}  \|\xi_0\|_{\mc{M}_x}\le M.
\end{align*}
In particular, up to taking an indistinguishable version, $\xi^\eps$ is a $\mc{M}_{x,M}$-valued solution (in the sense of Definition \ref{def:sol}). Moreover, for a.e.~$\omega$, we have: for every $t$, $\xi^\eps_t\ge 0$ (that is, it is a non-negative measure).
\end{lemma}

\begin{proof}
The non-negativity follows directly from the assumption  that $\xi_0$ is a non-negative measure  and so $\xi_0^\eps$ are non-negative functions and from the representation formula \eqref{eq:repr_formula}. Concerning the bound, this also follows from the representation formula \eqref{eq:repr_formula}, but we prefer giving a  short,  PDEs in spirit,    proof. Using $\varphi\equiv 1$ as test function in \eqref{eq:stochEulervort_test}, we get, for a.e. $\omega$: for every $t$,
\begin{align*}
\int_{\mt^2}\xi^\eps_t \rmd x = \int_{\mt^2}\xi^\eps_0 \rmd x.
\end{align*}
Since $\xi^\eps_t\ge 0$, we get that $\|\xi^\eps_t\|_{L^\infty_{t,\omega}(\mc{M}_x)}\le \|\xi_0\|_{\mc{M}_x}$ for every $t$, for a.e.~$\omega$. Defining $\xi^\eps=0$ outside the exceptional set where the bound is not satisfied, we get that $\xi^\eps$ is in $C_t(\mc{M}_{x,M})$ and so is a $\mc{M}_{x,M}$-valued solution.
\end{proof}

\begin{remark}
In the proof of the previous Lemma, we used the fact that $\xi^\eps$ is positive (and so $\int_\mt^2 \xi^\eps$ is the $\mc{M}_x$ norm of $\xi^\eps$), but this is not essential. Indeed, the vorticity equation is a transport equation with divergence-free velocity field, therefore the mass is conserved at least under suitable regularity assumption on the velocity, which are satisfied here.
\end{remark}

\subsubsection{Equation for the velocity and for its energy}
\label{subsect-velocity}

In view of a uniform $H^{-1}_x$ bound on $\xi^\eps$, we will:
\begin{trivlist}
 \item[ 1)] get an equation for the velocity $u^\eps=u^{\xi^\eps}=K \ast \xi^\eps$,
 \item[ 2)] get an equation for the energy of $u^\eps$, that is the $L^2_x$ norm of $u^\eps$,
 \item[ 3)] conclude a uniform $L^2_x$ bound on $u^\eps$ by Lemma \ref{lem:Green_function}, this bound is equivalent to a uniform $H^{-1}_x$ bound on $\xi^\eps$, up to the  average of $\xi^\eps$ on $\mt^2$.
\end{trivlist}

Here we consider a solution $\xi$ to the vorticity equation \eqref{eq:stochEulervort}, with sufficiently integrability to include the (bounded) approximants $\xi^\epsilon$. As we have seen in the introduction, see formula \eqref{eq:stochEulervel_intro}, in this case the corresponding velocity process $u$ is a solution to the following
\begin{align}
\begin{aligned}\label{eqn: Euler stoch velocity form}
&\partial_t u +(u\cdot\nabla) u +\sum_k(\sigma_k\cdot\nabla +(D\sigma_k)^{\mathrm{t}})u \circ \,\dot{W}^k = -\nabla p -\gamma,\\
&\divv u=0,
\end{aligned}
\end{align}
where processes $p:[0,T]\times\mt^2\times \Omega\to\mr$ and $\gamma:[0,T]\times\Omega \to \mr^2$ are also unknown. A  rigorous formulation of this result can be stated  as follows.

\begin{lemma}\label{lem:velocity_eq}  Assume that  $2<p<\infty$.  Assume  that a process    $\xi$ belongs to $L^p_{t,\omega}(L^p_x)$  and that
  $\xi$ is  an $\mc{M}_{x,M}$-valued distributional solution to the stochastic vorticity equation \eqref{eqn: Euler stoch vorticity form-Intro}.   Define a process $u$ by
  \[ u(t,\cdot,\omega) =K \ast \xi (t,\cdot,\omega), \;\; (t,\omega) \in [0,\infty)\times \Omega.\]
   Then $u$ belongs to  $L^p_{t,\omega}(\nH^{1,p}_x)$ and it is a distributional solution to the stochastic Euler equation, i.e.  for every $t\geq 0$, $\mathbb{P}-\text{a.s.}$
\begin{align}
\begin{aligned}\label{eqn-stochEulervel}
u_t &= u_0 -\int^t_0 \Pi[(u_r\cdot\nabla) u_r] \rmd r\\
&\ \ \ -\sum_k \int^t_0 \Pi[\sigma_k\cdot\nabla u_r +(D\sigma_k)^T u_r] \rmd W^k(r) \\
&\ \ \ +\frac{c}2 \int^t_0 \, \Delta u_r \rmd r,\quad \mbox{ in } H_x^{-1} .
\end{aligned}
\end{align}
 Here $\Pi: H_x^{-1} \to H_x^{-1}$ is the unique bounded extension of the Leray-Helmholtz  projection $\Pi:L^2(\mt^2,\mr^2) \to \rH$.
\end{lemma}
\textbf{Remark} Let us observe that the range $\Pi(H_x^{-1})$ is contained in  a closed subspace of   $H^{-1}_x$ consisting of  divergence-free zero-mean elements.

The proof is essentially based on the following heuristic identity, which rigorous proof is given in the Appendix.
If  the vector fields $v:\mt^2\to\mr^2$ and $w:\mt^2\to\mr^2$  are of sufficient regularity and $v$ is divergence-free, then
\begin{align*}
\curl [v\cdot\nabla w +(Dv)^{\mathrm{t}} w] = v\cdot\nabla \curl [w],
\end{align*}
where $Dv= \begin{pmatrix}                                D_1v_1 & D_1v_2 \\
                                D_2 v_1 & D_2 v_2
                              \end{pmatrix}$
 and ${}^{\mathrm{t}}$ is the transpose operation. Using the above, one can  pass from the velocity equation \eqref{eqn: Euler stoch velocity form} to the vorticity equation \eqref{eqn: Euler stoch vorticity form-Intro}.
From the equation for  the velocity  \eqref{eqn-stochEulervel} we can deduce  the equation for the expected valued of the energy:

\begin{lemma}\label{lem:u_norm}
Under the assumptions of Lemma \ref{lem:velocity_eq},
\begin{align}
\label{eq:u_norm}
\E\|u_t\|_{L^2_x}^2 &=
\E\|u_0\|_{L^2_x}^2 -2\E \int^t_0 \int_{\mt^2} u\cdot \Pi[(u \cdot\nabla) u] \rmd x \rmd r \\
\nonumber
& -c \E\int^t_0 \int_{\mt^2} |\nabla u|^2 \rmd x \rmd r \\
& +\E \int^t_0 \sum_k \int_{\mt^2} |\Pi[(\sigma_k\cdot\nabla +(D\sigma_k)^{\mathrm{t}}) u]|^2 \rmd x \rmd r,\;\;\;t\geq 0.
\nonumber
\end{align}
\end{lemma}

One can get this equation formally from \eqref{eqn: Euler stoch velocity form} by applying the It\^o formula to $\|u\|_{L^2_x}^2$. However this is not possible rigorously, because the rigorous equation \eqref{eqn-stochEulervel} holds in $H^{-1}_x$ and the square of the $L^2_x$ norm is not continuous on $H^{-1}_x$. The rigorous proof of the formula \eqref{eq:u_norm} is based on a regularization argument and is postponed to the Appendix \ref{app-B}.

\subsubsection{Bound in $L^2_{t,\omega}(H^{-1}_x)$}

Now we give a uniform $L^2_{t,\omega}(H^{-1}_x)$ bound on the approximants $\xi^\eps$. This bound is not essential for the compactness argument, but it is essential for the convergence argument, as we will see.

\begin{lemma}\label{lem-Hm1_bound}
There exist $C>0$ such that for all $\eps \in (0,1]$ the following inequality  holds
\begin{align*}
\E \|\xi^\eps_t\|_{H^{-1}_x}^2\le C (\|\xi_0\|_{H^{-1}_x}^2 +\|\xi_0\|_{\mc{M}_x}^2),\quad \text{for every }t.
\end{align*}
\end{lemma}

In the proof of this lemma, we use crucially the assumptions \ref{assumption_sigma} on $\sigma_k$. We recall that $u^\eps=u^{\xi^\eps}=K \ast \xi^\eps$.

\begin{proof}
By Remark \ref{rmk:Borel_norm}, the $H^{-1}_x$ norm is a Borel function on $(\mc{M},w^\ast )$, therefore $\|\xi^\eps\|_{H^{-1}_x}$ and $\|u^\eps\|_{L^2_x}$ are progressively measurable and the expectations of their moments make sense. By Lemma \ref{lem:Green_function}, applied to $\xi^\eps_t-\int \xi^\eps_t(y) dy$, we have, for every $t$,
\begin{align*}
&\|\xi^\eps_t\|_{H^{-1}_x} \le C\|u^\eps_t\|_{L^2_x} + \left|\int_{\mt^2}\xi^\eps_t dx\right|,\\
&\|u^\eps_t\|_{L^2_x}\le C\|\xi^\eps_t\|_{H^{-1}_x}.
\end{align*}
By Lemma \ref{lem:cons_mass}, the $L^1$ norm of $\xi^\eps$ is uniformly bounded by $\|\xi_0\|_{\mc{M}_x}$. Hence it is enough to show, for every $t$,
\begin{align*}
\E \|u^\eps_t\|_{L^2_x}^2\le C\|u_0\|_{L^2_x}^2.
\end{align*}
We will show the above bounds by  using the velocity equation \eqref{eqn-stochEulervel}.

We start with equation \eqref{eq:u_norm} applied to $u^\eps$. Since $u^\eps$ is divergence-free, the nonlinear term in \eqref{eq:u_norm} vanishes. Indeed, we have the following train of equalities,
\begin{align*}
&\int_{\mt^2} u^\eps\cdot \Pi[(u^\eps \cdot\nabla) u^\eps] \rmd x = \int_{\mt^2} \Pi[u^\eps]\cdot (u^\eps \cdot\nabla) u^\eps \rmd x\\
&= \int_{\mt^2} u^\eps\cdot (u^\eps \cdot\nabla) u^\eps \rmd x = 0.
\end{align*}
For the term with $\sigma_k$, since the Leray-Helmholtz projection $\Pi$ is a contraction  in the space $L^2_x$, we infer that
\begin{align*}
&\sum_k \int_{\mt^2} |\Pi[(\sigma_k\cdot\nabla +(D\sigma_k)^{\mathrm{t}}) u^\eps]|^2 \rmd x \le \sum_k \int_{\mt^2} |(\sigma_k\cdot\nabla +(D\sigma_k)^{\mathrm{t}}) u^\eps|^2 \rmd x\\
&= \int_{\mt^2} [\sum_k|\sigma_k \cdot\nabla u^\eps|^2 +\sum_k|(D\sigma_k)^{\mathrm{t}} u^\eps|^2 +2\sum_{i,j,h} \sum_k\sigma_k^i \partial_{x_j} \sigma_k^h (u^\eps)^h \partial_{x_i} (u^\eps)^j ] \rmd x.
\end{align*}
Now we use the assumptions \ref{assumption_sigma}, precisely that $\sum_k\sigma_k^i\sigma_k^j = c\delta_{ij}$ and that $\sum_k\sigma_k^i \partial_{x^j} \sigma_k^h = 0$ for all $i,j,h$, with uniform (with respect to $x$) convergence in the series over $k$: we get
\begin{align*}
&\sum_k \int_{\mt^2} |\Pi[(\sigma_k\cdot\nabla +(D\sigma_k)^{\mathrm{t}}) u^\eps]|^2 \rmd x \le \int_{\mt^2} c|\nabla u^\eps|^2 \rmd x +\sum_k\|\sigma_k\|_{C^1_x} \int_{\mt^2} |u^\eps|^2 \rmd x .
\end{align*}
Putting all together, we obtain for every $t$
\begin{align*}
\E\|u^\eps_t\|_{L^2_x}^2 \le  \E\|u^\eps_0\|_{L^2_x}^2 +\sum_k\|\sigma_k\|_{C^1_x} \int^t_0 \E\|u^\eps_r\|_{L^2_x}^2 \rmd r.
\end{align*}
By applying the  Gronwall Lemma to the function $\E\|u^\eps_t\|_{L^2_x}^2$,  we conclude, that
\begin{align*}
\E\|u^\eps_t\|_{L^2_x}^2 \le  \E\|u^\eps_0\|_{L^2_x}^2 \exp[t\sum_k\|\sigma_k\|_{C^1_x}], \;\;\; t\geq 0.
\end{align*}
This  implies the desired bound since $u^\eps_0$ is deterministic and $\|u^\eps_0\|_{L^2_x}\le C\|u_0\|_{L^2_x}$. The proof of Lemma \ref{lem-Hm1_bound} is complete.
\end{proof}

\subsubsection{Bound in $L^m_\omega(C_t^\alpha(H^{-4}_x))$}

Now we prove a uniform $L^m_\omega(C_t^\alpha(H^{-4}_x))$ bound, for $m\ge 2$:

\begin{lemma}\label{lem:Hm4_bound}
Assume that $T>0$,  $ m \in [2,\infty)$ and  $0<\alpha<1/2$. Then   there  exists a constant $C=C(m,\alpha,T,M)$ such that for all $\eps \in (0,1]$,

\begin{align*}
\E \| \xi^\eps \|_{C^\alpha_t(H^{-4}_x)}^m \le C (\|\xi_0\|_{\mc{M}_x}^{2m} +\|\xi_0\|_{\mc{M}_x}^m).
\end{align*}
\end{lemma}

\begin{proof} Let us  choose and fix $\alpha \in (0,\frac12)$.  Choose then $m$ such that $\alpha<\frac12-\frac1m$.

Note that, by Remark \ref{rmk:Borel_norm}, $\omega\mapsto \|\xi^\eps\|_{C_t^\alpha(H^{-4}_x)}$ is measurable. Using the equation \eqref{eq:stochEulervort}, we get, for all $s, t \in [0,T]$ such that $s\leq t$,
\begin{align*}
&\E \| \xi^\eps_t - \xi^\eps_s \|_{H^{-4}_x}^m\\
&\le C\E \left\| \int^t_s N(\xi^\eps_r) \rmd r \right\|_{H^{-4}_x}^m\\
&\ \ \ +C\E \left\| \sum_k \int^t_s \sigma_k\cdot\nabla \xi^\eps_r \rmd W^k \right\|_{H^{-4}_x}^m\\
&\ \ \ +Cc^m\E \left\| \int^t_s \Delta \xi^\eps_r \rmd r \right\|_{H^{-4}_x}^m.
\end{align*}
By Lemma \ref{lem:Poupaud_trick}, we get the following inequality for the   for the nonlinear term
\begin{align*}
\E \left\| \int^t_s N(\xi^\eps_r) \rmd r \right\|_{H^{-4}_x}^m \le C(t-s)^m \|\xi^\eps\|_{L^\infty_{t,\omega}(\mc{M}_x)}^{2m}.
\end{align*}
What concerns  the stochastic integral, by the Burkholder-Davis-Gundy inequality and Lemma \ref{lem:H_Borel}  we infer that
\begin{align*}
&\E \left\| \sum_k \int^t_s \sigma_k\cdot\nabla \xi^\eps_r \circ \rmd W^k \right\|_{H^{-4}_x}^m\\
&\le C \E \left( \sum_k \int^t_s \|\sigma_k\cdot\nabla \xi^\eps_r \|_{H^{-4}_x}^m \rmd r \right)^{m/2}\\
&\le C(t-s)^{m/2} \left(\sum_k\|\sigma_k\|_{C_x}^2\right)^{m/2} \|\xi^\eps\|_{L^\infty_{t,\omega}(\mc{M}_x)}^m
\end{align*}
Finally,  for the second order term,  again by Lemma \ref{lem:H_Borel}  we deduce that
\begin{align*}
\E \left\| \int^t_s \Delta \xi^\eps_r \rmd r \right\|_{H^{-4}_x}^m\le C \|\xi^\eps\|_{L^\infty_{t,\omega}(\mc{M}_x)}^m (t-s)^m
\end{align*}
We put all together and we recall the a priori bound on $\|\xi^\eps\|_{L^\infty_{t,\omega}(\mc{M}_x)}$ in Lemma \ref{lem:cons_mass}: we obtain
\begin{align*}
\E \| \xi^\eps_t - \xi^\eps_s \|_{H^{-4}_x}^m \le C (t-s)^{m/2} (\|\xi_0\|_{\mc{M}_x}^m +\|\xi_0\|_{\mc{M}_x}^{2m}),\;\; 0\leq s\leq t\leq T,
\end{align*}
where the constant $C$ depends on $\sum_k\|\sigma_k\|_{B_x}^2$ and on $c$. By the Kolmogorov criterion, or the Sobolev embedding  as in \cite{DaPZab2014},  recalling that $\xi^\eps$ is already continuous as $H^{-4}_x$-valued process, we infer that
\begin{align*}
\E \| \xi^\eps \|_{C^\alpha_t(H^{-4}_x)}^m \le C(\|\xi_0\|_{\mc{M}_x}^m +\|\xi_0\|_{\mc{M}_x}^{2m}).
\end{align*}
The proof is complete.
\end{proof}

\subsection{Tightness}
Let us recall  that $M>0$ is fixed such that $\|\xi_0\|_{\mc{M}_x}\le M$. We also fix $T>0$, see Remark \ref{rem-T}. Whenever we use subscript ${}_t$ we mean that the corresponding functions are defined on the closed time interval $[0,T]$.

In this section we prove the tightness of laws on $C_t(\mc{M}_{x,M},w^\ast ) \cap (L^2_t(H^{-1}_x),w)$ of the family of processes  $(\xi^\eps)_{\eps \in (0,1]}$.

We recall that $(\mc{M}_{x,M},w^\ast )$ is metrizable with the distance $d_{\mc{M}_{x,M}}(\mu,\nu) = \sum_j 2^{-j} |\lan \mu-\nu, \varphi_j \ran|$, therefore $C_t(\mc{M}_{x,M},w^\ast )$ is metrizable as well, see Remark \ref{rmk:continuity_weak*} with $X=C_x$. The intersection space $C_t(\mc{M}_x,w^\ast )\cap (L^2_t(H^{-1}_x),w)$ is defined as the set consisting of those elements  of $C_t(\mc{M}_x,w^\ast )$ which have  finite $L^2_t(H^{-1}_x)$ norm.
On this set, which is a subspace of $C_t(\mc{M}_x,w^\ast )$, the topology is induced by the $C_t(\mc{M}_x,w^\ast )$  and the $(L^2_t(H^{-1}_x),w)$ topologies, i.e.
it is the weakest topology such that the natural
embeddings to  $C_t(\mc{M}_x,w^\ast )$  and  $(L^2_t(H^{-1}_x),w)$ are continuous. It follows that a function from another topological space with values in the intersection space is continuous if and only if it is composition with those two natural embeddings are continuous. Moreover, one can show that a function from another measurable  space with values in the intersection space is Borel measurable if and only if
it is composition with those two natural embeddings are Borel measurable.

Let us finish by pointing out that by Lemma \ref{lem:equiv_topol},  the Borel $\sigma$-algebra generated by $(L^2_t(H^{-1}_x),w)$ coincides with the Borel $\sigma$-algebra generated by the strong topology on $L^2_t(H^{-1}_x)$,

Note that, by Lemmas \ref{lem:cons_mass} and \ref{lem-Hm1_bound}, for any $\eps>0$, $\xi^\eps_t$ takes values in $\mc{M}_{x,M}\cap H^{-1}_x$ for every $t$, $\mathbb{P}$-a.s. and so for every $\omega$ up to taking an indistinguishable version. Hence, by Lemma \ref{lem:C_Lebesgue_meas}, $\xi$ is Borel measurable as $C_t(\mc{M}_{x,M},w^\ast )$-map and as $(L^2_t(H^{-1}_x),w)$-valued maps, i.e. both spaces being endowed with their Borel $\sigma$-algebras. Hence, by the above we infer that   $\xi^\eps$ is Borel measurable as a  $C_t(\mc{M}_{x,M},w^\ast ) \cap (L^2_t(H^{-1}_x),w)$-valued map, i.e.
$\xi^\eps$ is a  $C_t(\mc{M}_{x,M},w^\ast ) \cap (L^2_t(H^{-1}_x),w)$-valued random variable.

\begin{lemma}\label{lem:tightness}
 The  family of laws of $(\xi^\eps)_{\eps \in (0,1]}$ is tight on $C_t(\mc{M}_{x,M},w^\ast ) \cap (L^2_t(H^{-1}_x),w)$.
\end{lemma}

We start with a generalization of \cite[Lemma 3.1]{BrzMot2013}. The latter is a refined version of the compactness argument in \cite{FlaGat1995}, which can be seen as a stochastic version of the Aubin-Lions lemma. Given a Banach space $X$, we call $B_M^X$ the closed ball in $X$ of radius $M$.

\begin{lemma}\label{lem-compact}
Let $X$, $Y$ be separable Banach spaces with $Y$ densely embedded in $X$. Then, for every $M\ge0$, $\alpha>0$, $a\ge0 $, the set
\begin{align*}
A_a = \{ z \in C_t(B_M^{X^\ast},w^\ast ) : \|z\|_{C^\alpha_t(Y^\ast)}\le a \}
\end{align*}
is compact in $C_t(B_M^{X^\ast},w^\ast )$.
\end{lemma}

\begin{remark}\label{rmk:continuity_weak*}
For the proof, we recall the following facts. First, the ball $B_M^{X^\ast}$ endowed with the weak-$^\ast$ topology is metrizable with the distance $d_{B_M^{X^\ast}}(w,w') = \sum_j 2^{-j} |\lan w-w', \varphi_j \ran|$, where $(\varphi_j)_j$ is a dense sequence in $B^X_1$, see \cite[Theorem 3.28]{Brezis-2011}. Hence the set $C_t(B_M^{X^\ast},w^\ast )$ is metrizable with the distance
\begin{align}
d(z,z') = \sup_{t\in[0,T]} \sum_j 2^{-j} |\lan z-z', \varphi_j \ran|,\quad z,z'\in C_t(B_M^{X^\ast},w^\ast ).\label{eq:dist_CM}
\end{align}
Moreover, given an $C_t(B_M^{X^\ast},w^\ast )$-valued  sequence $(z^n)_n$, an element  $z \in C_t(B_M^{X^\ast},w^\ast )$ and a dense subset $D \subset X$, the following three conditions are equivalent.
\begin{trivlist}
\item[$\bullet$] the sequence $z^n$ converges to $z$ in $C_t(B_M^{X^\ast},w^\ast )$;
\item[$\bullet$] for every $\varphi \in X$, the sequence $\lan z^n,\varphi \ran$  in $C_t$ to $\lan z,\varphi \ran$;
\item[$\bullet$] for every $\varphi \in D$, the sequence $\lan z^n,\varphi \ran$ converges  in $C_t$ to $\lan z,\varphi \ran$.
\end{trivlist}
Equivalence between the first two conditions can be seen using the distance defined in \eqref{eq:dist_CM}. Equivalence  between the last two conditions can be seen by approximating a generic element  $\varphi \in X$ by an $D$-valued sequence  and using the uniform bound $\sup_n \sup_{t\in [0,T]} \|z_n\|_{X^\ast}\le M$.
\end{remark}

\begin{proof}[Proof of Lemma \ref{lem-compact}]
Since the topological space  space $C_t(B_M^{X^\ast},w^\ast )$ is metrizable, the compactness is equivalent to the sequential compactness. Let $(z^n)_n$ be a sequence in $A_a$. We have to find a subsequence $(z^{n_k})_k$ which converges in $C_t(B_M^{X^\ast},w^\ast )$ to an element of the set  $A_a$.

For fixed $t$ in $\mathbb{Q}\cap [0,T]$, $(z^n_t)_n$ is an $B^{X^\ast}_M$-valued sequence.  Hence, by the Banach-Alaoglu Theorem, there exists a subsequence $(z^{n_k}_t)_k$ which is convergent weakly-$\ast$ to an element $\tilde{z}_t \in B^{X^\ast}_M$. By a diagonal procedure, we can find  a sequence $(n_k)_k$ independent of $t$ in $\mathbb{Q}\cap [0,T]$.

On the other side, let $D$ be a countable dense set in $Y$, and so in $X$. The fact that $z^n$ are equicontinuous and equibounded in $Y^\ast$ implies that, for every $\varphi \in D$, the functions $t\mapsto \lan z^{n_k}_t,\varphi \ran$ are equicontinuous and equibounded, their $C^\alpha$ norm being bounded by $a\|\varphi\|_Y$. Hence, by the Ascoli-Arzel\`a Theorem, there exists a subsequence converging  in $C_t$ to some element $f^\varphi=\{[0,T] \ni t\mapsto f^\varphi_t\}$, which also satisfies $\|f^\varphi\|_{C^\alpha_t}\le a\|\varphi\|_Y$. By a diagonal procedure, we can choose the subsequence independent of $\varphi \in D$. With a small abuse of notation, we continue using $n_k$ for this subsequence. Then, for all $t$ in $\mathbb{Q}\cap [0,T]$, for all $\varphi \in D$, $\lan \tilde{z}_t,\varphi\ran = f^\varphi_t$.

Fix $t$ in $[0,T]$ and let $(t_j)_j$ be a sequence in $\mathbb{Q}\cap [0,T]$ converging to $t$. Since the sequence $(\tilde{z}_{t_j})_j$ takes values in the ball  $B^{X^\ast}_M$,  up to a subsequence, it converges weakly-$\ast$ to an element $z_t$ in $B^{X^\ast}_M$. On the other hand, for every $\varphi \in D$, by the continuity of the function $[0,T]\ni t\mapsto f^\varphi_t$, we infer that  $\lan z_t,\varphi \ran = f^\varphi_t$.

Since the map $[0,T]\ni t\mapsto \lan z_t,\varphi \ran = f^\varphi_t$ is continuous for every $\varphi \in D$, and in fact for every $\varphi \in X$, by an approximation argument, we infer that  $z=[0,T]\ni t \mapsto z_t $  belongs to  $C_t(B^{X*}_M,w^\ast )$. Moreover
\begin{align*}
&\|z_t-z_s\|_{Y^\ast} = \sup_{\varphi \in D,\|\varphi\|_Y\le 1} |\lan z_t-z_s,\varphi \ran|\\
&= \sup_{\varphi \in D,\|\varphi\|_Y\le 1} |f^\varphi_t-f^\varphi_s| \le a|t-s|^\alpha,
\end{align*}
and similarly for $\|z_t\|_{Y^\ast}$ alone. Hence $\|z\|_{C^\alpha_t(Y^\ast)}\le a$ and so $z$ is in $A_a$.

Finally, for every $\varphi \in D$, $\lan z^{n_k},\varphi \ran$ converges uniformly to $f^\varphi = \lan z,\varphi \ran$. Therefore, by Remark \ref{rmk:continuity_weak*}, we infer that $z^{n_k}$ converges to $z$ in $C_t((B_M^{X^\ast},w^\ast )$. The proof is complete.
\end{proof}

%
%
%

As a consequence of the previous Lemma and the Banach-Alaoglu Theorem, we get the following result.

\begin{lemma}\label{lem:cpt_set}
For  all  $M>0$, $\alpha>0$ and  $a,b\ge0 $, the set
\begin{align*}
A_{a,b}=A_{a,b}(M,\alpha) = \{ \mu \in C_t(\mc{M}_{x,M},w^\ast ) \cap L^m_t(H^{-1}_x) : \|\mu\|_{C^\alpha_t(H^{-4}_x)}\le a, \ \|\mu\|_{L^m_t(H^{-1}_x)}\le b \}
\end{align*}
is metrizable and compact subset of  $C_t(\mc{M}_{x,M},w^\ast ) \cap (L^m_t(H^{-1}_x),w)$.
\end{lemma}

\begin{proof}
Since the topologies on $C_t(\mc{M}_{x,M},w^\ast )$ and on the closed ball of radius $b$ in $(L^m_t(H^{-1}_x),w^\ast )$ are metrizable, $A_{a,b}$ is metrizable as well and the compactness is equivalent to the sequential compactness.

Let $(\mu^n)_n$ be a sequence in $A_{a,b}$. By the previous Lemma, applied to $X=C_x$ and $Y=H^4_x$, there exists a sub-subsequence $(\mu^{n_k})_k$ converging to some $\mu$ in $C_t(\mc{M}_{x,M},w^\ast )$ with $\|\mu\|_{C^\alpha_t(H^{-4}_x)}\le a$. On the other hand, by the Banach-Alaoglu theorem, there exists a subsequence, which we can assume $(\mu^{n_k})_k$ up to relabelling, converging to some $\nu$ in $(L^m_t(H^{-1}_x),w)$ with $\|\nu\|_{L^m_t(H^{-1}_x)}\le b$. Using these two limits, for every $g$ in $C_t$ and every $\varphi \in C^1_x$, we have
\begin{align*}
\int^T_0 g(t) \lan \mu_t, \varphi \ran \rmd t = \int^T_0 g(t) \lan \nu_t, \varphi \ran \rmd t.
\end{align*}
Hence $\mu=\nu$ and so $\mu$ is the limit in $A_{a,b}$ of the subsequence $(\mu^{n_k})_k$. The proof is complete.
\end{proof}

We are ready to prove tightness of $\xi^\eps$.

\begin{proof}[Proof of Lemma \ref{lem:tightness}]
As we have seen at the beginning of this section, by Lemmas \ref{lem:cons_mass} and \ref{lem-Hm1_bound}, for any $\eps>0$, $\xi^\eps$ is, up to an indistinguishable version, a $C_t(\mc{M}_x,w^\ast )\cap (L^2_t(H^{-1}_x),w)$-valued random variable. Lemma \ref{lem:cpt_set} ensures that the set $A_{a,b}$ defined in that Lemma is metrizable and compact in $C_t(\mc{M}_x,w^\ast )\cap (L^2_t(H^{-1}_x),w)$. The Markov inequality gives
\begin{align*}
&P\{\xi^\eps \notin A_{a,b} \}\le P\{ \| \xi^\eps \|_{C_t^\alpha(H^{-4}_x)}>a \} + P\{ \| \xi^\eps \|_{L^m_t(H^{-1}_x)}>b \}\\
&\le a^{-m} \E \| \xi^\eps \|_{C_t^\alpha(H^{-4}_x)}^m + b^{-m}\E \| \xi^\eps \|_{L^m_t(H^{-1}_x)}^m.
\end{align*}
By Lemmas \ref{lem-Hm1_bound} and \ref{lem:Hm4_bound}, the right-hand side above can be made arbitrarily small, uniformly in $\eps$, taking $a$ and $b$ large enough. The tightness is proved.
\end{proof}

As a consequence, we have actually:

\begin{corollary}\label{cor:tightness_couple}
The family $(\xi^\eps,W)_{\eps \in (0,1]}$, where $W=(W^k)_k$ is tight on the space $\chi := [C_t(\mc{M}_{x,M},w^\ast )\cap(L^m_t(H^{-1}_x),w)] \times C_t^{\mathbb{N}}$.
\end{corollary}

\begin{proof}
The tightness of $(\xi^\eps,W)$ follows easily from the tightness of the marginals.
\end{proof}

\subsection{Convergence}

Let us begin this subsection with an observation that we can apply the Skorohod-Jakubowski representation theorem, see \cite{Jak1997} and \cite{Brz+Ondr_2011} to the family $(\xi^\eps,W)_\eps$ and the space
\[
\Xi = [C_t(\mc{M}_{x,M},w^\ast ) \cap (L^m_t(H^{-1}_x),w)] \times C_t^\mathbb{N}.\]
 Indeed the family $(\xi^\eps,W)_\eps$ is tight by Corollary \ref{cor:tightness_couple} and the space $\Xi$ satisfies the assumption (10) in \cite{Jak1997}: for given sequences $(t_i)_i$ dense in $[0,T]$, $(\varphi_j)_j$ dense in $C_x$, the maps $f_{i,j}$ and $g_{i,k}$, defined on $\Xi$ by $f_{i,j}(\mu,\gamma)=\lan \mu_{t_i},\varphi_j\ran$ and $g_{i,k}(\mu,\gamma)= \arctan(\gamma^k_{t_i})$, form a sequence of continuous, uniformly bounded maps separating points in $\Xi$.

Hence, by the Skorohod-Jakubowski Theorem there exist an infinitesimal sequence $(\eps_j)_j$, a probability space $(\td{\Omega},\td{\mc{A}},\td{\mathbb{P}})$, a $\Xi$-valued sequence $(\td{\xi}^j, \td{W}^{(j)})_j$ and a $\Xi$-valued random variable $(\td{\xi},\td{W})$ such that $(\td{\xi}^j,\td{W}^{(j)})$ has the same law of $(\xi^{\eps_j},W)$ and $(\td{\xi}_j,\td{W}^{(j)})$ converges to $(\td{\xi},\td{W})$ a.s. in $C_t(\mc{M}_{x,M},w^\ast ) \cap (L^m_t(H^{-1}_x),w)$. For notation, we use $\td{W}^{(j),k}$ and $\td{W}^k$ for the $k$-th component of the $C_t^{\mathbb{N}}$-valued random variables $\td{W}^{(j)}$ and $\td{W}$.

Let us denote  $\td{\mathbb{F}^0}=(\td{\mc{F}}_t^{0})_{t\in [0,T]}$ the filtration generated by the processes $\td{\xi}$, $\td{W}$ and the $\td{\mathbb{P}}$-null sets on $(\td{\Omega},\td{\mc{A}},\td{\mathbb{P}})$. We also put $\td{\mathbb{F}}=(\td{\mc{F}}_t)_{t\in [0,T]}$, where  $\td{\mc{F}}_t = \cap_{s>t}\td{\mc{F}}^{0}_s$. Similarly, we denote by  $\td{\mathbb{F}^{0,j}}=( \td{\mc{F}}_t^{0,j})_{t\in [0,T]}$ the filtration generated by the processes $\td{\xi}^j$, $\td{W}^{(j)}$ and the $\td{\mathbb{P}}$-null sets on $(\td{\Omega},\td{\mc{A}},\td{\mathbb{P}})$ and finally we put $\td{\mathbb{F}^j}=( \td{\mc{F}}_t^{j})_{t\in [0,T]}$, where  $\td{\mc{F}}^j_t = \cap_{s>t}\td{\mc{F}}^{0,j}_s$.

\begin{lemma}\label{lem:BM_enlarged}
The filtration $\td{\mathbb{F}}$ is complete and right-continuous and $\td{W}$ is a cylindrical $\td{\mathbb{F}}$-Wiener process. Moreover, the process  $\td{\xi}$ is an $(\mc{M}_{x,M},w^\ast )$-valued $(\td{\mc{F}}_t)_{t\in [0,T]}$-progressively measurable. Similar statements hold for  $\td{\mathbb{F}}^j$, $\td{W}^{(j)}$ and $\td{\xi}^j$, for every $j \in \mathbb{N}$.
\end{lemma}

The proof of this Lemma is simple but technical and postponed to  Appendix \ref{app-B}.

Now we will prove that each copy $\td{\xi}^j$ of the approximating process  $\xi^{\eps_j}$ is a solution to the stochastic vorticity equation.

\begin{lemma}\label{lem:eq_enlarged}
For every $j\in \mathbb{N}$, the object \[(\td{\Omega},\td{\mc{A}},\td{\mathbb{F}}^j,\td{\mathbb{P}},\td{W}^{(j)},\td{\xi}^j)\]
is a $\mc{M}_{x,M}$-valued solution to the vorticity equation \eqref{eqn: Euler stoch vorticity form-Intro} with the initial condition $\td{\xi}^j_0=\xi_0^{\eps_j}$ $\mathbb{P}$-a.s.. Moreover Lemmata \ref{lem-Hm1_bound} and \ref{lem:Hm4_bound} hold for $\td{\xi}^j$ in place of $\xi^{\eps}$ and, $\mathbb{P}$-a.s.,~$\td{\xi}_t$ is non-negative for every $t\in [0,T]$.
\end{lemma}

Also the proof of this Lemma is technical and postponed to Appendix \ref{app-B}.

\subsection{Limiting equation}

Now we show that $\td{\xi}$ satisfies the vorticity equation with $\td{W}$ as Brownian motion. With this Lemma, Theorem \ref{thm:main} is proved.

\begin{lemma}\label{lem:limit_sol}
The object $(\td{\Omega},\td{\mc{A}},(\td{\mc{F}}_t)_{t\in [0,T]},\td{\mathbb{P}},\td{W},\td{\xi})$ is a $\mc{M}_{x,M}$-valued solution to the vorticity equation \eqref{eqn: Euler stoch vorticity form-Intro}, which is also in $C_t(\mc{M}_{x,M},w^\ast ) \cap (L^2_t(H^{-1}_x),w)$.
\end{lemma}


\begin{proof}[Proof of Lemma \ref{lem:limit_sol}] The proof of Lemma \ref{lem:limit_sol} will occupy the remaining parts of the current subsection.

To prove Lemma \ref{lem:limit_sol} we will show \eqref{eq:stochEulervort_test} for $(\td{\xi},\td{W})$ for every test function $\varphi \in C^\infty_x$. By Lemma \ref{lem:eq_enlarged}, for each $j$, $(\td{\xi}^j,\td{W}^{(j)})$ satisfies \eqref{eq:stochEulervort_test} for every $\varphi \in C^\infty_x$. Hence it is enough to pass to the $\td{\mathbb{P}}$-a.s.~limit, as $j\to \infty$, in each term of \eqref{eq:stochEulervort_test} for $(\td{\xi}^j,\td{W}^{(j)})$, possibly choosing a subsequence, for every $t$ and every $\varphi \in C^\infty_x$. We fix $t$ in $[0,T]$ and $\varphi \in C^\infty_x$.

We start with the deterministic linear terms: $\lan \td{\xi}^j_t, \varphi \ran$, $\lan \td{\xi}^j_0, \varphi \ran$ and
\begin{align*}
\int^t_0 \lan \td{\xi}^j_r, c\Delta \varphi \ran \rmd r
\end{align*}
converge $\mathbb{P}$-a.s.~to the corresponding terms without the superscript $j$, thanks to the convergence of $\td{\xi}^j$ to $\td{\xi}$ in $C_t(\mc{M}_{x,M},w^\ast )$.

\subsubsection{The nonlinear term}

Concerning the nonlinear term, we recall Lemma \ref{lem:Poupaud_trick}) and we follow the Schochet argument, see Schochet \cite{Sch1995} and Poupaud \cite[Section 2]{Pou2002}. The first main ingredient for the convergence is the following:

\begin{lemma}\label{lem:continuity_nonlin}
Fix $M>0$. For every $\varphi \in C^2_x$, the map $\mu\mapsto \lan N(\mu),\varphi \ran$ is continuous on the subset $\mc{M}_{x,M,+,\mathrm{no-atom}}$ of $\mc{M}_{x,M}$ of non-negative non-atomic measures with total mass bounded by $M$, endowed with the weak-$^\ast$ topology.
\end{lemma}

We use the following result, a version of the classical Portmanteau theorem (which deals with probability measures rather than non-negative measures):

\begin{lemma}\label{lem:continuity_mass}
Let $X$ be a compact metric space. Assume that $(\nu^k)_k$ is a sequence of non-negative bounded measures and converges to $\nu$ in $(\mc{M}(X),w^\ast )$. Let $F$ be a closed set in $X$ with $\nu(F)=0$ and let $\psi:X\to\mr$ be a bounded Borel function, continuous on $X\setminus F$. Then the sequence $(\lan \nu^k,\psi \ran)_k$ converges to $\lan \nu,\psi \ran$.
\end{lemma}

\begin{proof}
Let $\eps>0$, we have to prove that $|\lan \nu^k-\nu, \psi \ran| <C\eps$ for $k$ large enough, for some constant $C$. The fact that $\nu(F)=0$ implies the existence of $\delta>0$ such that $\nu(\bar{B}(F,\delta))<\eps$, where $\bar{B}(F,\delta):=\{x\in X: d(x,F)\le \delta\}$. As the function $1_{\bar{B}(F,\delta)}$ is upper semi-continuous and $(\nu^k)_k$ converges weakly-$\ast$ to $\nu$, there exists $\bar{k}$ such that $\nu^k(\bar{B}(F,\delta)) <\eps$ for all $k\ge \bar{k}$. By Urysohn lemma, there exists a continuous function $\rho$ with $0\le \rho\le 1$, $\rho=1$ on $F$ and $\rho=0$ on $\bar{B}(F,\delta)^c$; it is easy to see that $\psi(1-\rho)$ is then continuous on all $X$. Now we split
\begin{align}
|\lan \nu^k-\nu, \psi \ran| \le |\lan \nu^k, \psi\rho \ran| +|\lan \nu^k-\nu, \psi(1-\rho) \ran| +|\lan \nu, \psi\rho \ran|.\label{eq:ineq_Pormanteau}
\end{align}
For the first term in the right-hand side, we have
\begin{align} \label{eq:positivity}
|\lan \nu^k, \psi\rho \ran|&\le \sup_k |\nu^k(\bar{B}(F,\delta))| \sup_X|\psi|
\\
&= \sup_k \nu^k(\bar{B}(F,\delta)) \sup_X|\psi| \le \eps \sup_X|\psi|.
\nonumber
\end{align}
The same inequality holds for the third term in the right-hand side of \eqref{eq:ineq_Pormanteau}. Finally, the second term in \eqref{eq:ineq_Pormanteau} is bounded by $\eps$ provided $k$ is large enough, by weak-$^\ast$ convergence of $(\nu^k)_k$. The proof is complete.
\end{proof}

\begin{remark}
It is only in \eqref{eq:positivity} in the proof of Lemma \ref{lem:continuity_mass} that we had  to use that the fact that the process $\xi$ takes values in non-negative measures.
\end{remark}

\begin{proof}[Proof of Lemma \ref{lem:continuity_nonlin}]
We have to show that, for every sequence $(\mu^n)_n$ converging to $\xi \in \mc{M}_{x,M,+,\mathrm{no-atom}}$, the sequence $\lan N(\mu^n),\varphi\ran$ converges to $\lan N(\mu),\varphi\ran$. By Lemma \ref{rmk:cont_prod_meas} in the Appendix, $(\mu^n\otimes\mu^n)_n$ converges weakly-$\ast$ to $\xi\otimes \xi$. Moreover, since $\mu$ has no atoms, then $\mu \otimes \mu$ gives no mass to the diagonal $D=\{(x,y): x=y\}$. Indeed,  by the Fubini theorem,
\begin{align*}
(\mu\otimes\mu)(D) = \int_{\mt^2}\mu(\rmd x) \int_{\{x\}}\mu(\rmd y) = 0.
\end{align*}
We are now in a position to apply Lemma \ref{lem:continuity_mass} to the sequence $(\mu^n\otimes\mu^n)_n$, with the state space $X=\mt^2 \times \mt^2 $, with $F=D$ and with $\psi=F_\varphi$, which is continuous outside the diagonal $D$: we get that
\begin{align*}
\lan \mu^n\otimes\mu^n , F_\varphi \ran \to \lan \mu\otimes\mu , F_\varphi \ran,
\end{align*}
which is exactly the desired convergence. The proof is complete.
\end{proof}

The second ingredient for the convergence of the nonlinear term is the following:

\begin{lemma}\label{lem:no_atoms}
Let $\mu \in \mc{M}_x \cap H^{-1}_x$. Then the measure $\mu$ has no atoms.
\end{lemma}

\begin{proof}[Proof of Lemma \ref{lem:no_atoms}]
Fix $x_0$ in $\mt^2$, we have to prove that $\mu(\{x_0\})=0$. Let $\rho:\mr^2\to\mr$ be a smooth function with $0\le \rho \le 1$, supported on $B_1(0)$ (the ball centered at $0$ with radius $1$) and with $\rho(x)=1$ if and only if $x=0$. For $n$ positive integer, call $\rho_n(x)=\rho(n(x-x_0))$ and take its periodic version on $\mt^2$, which, with small abuse of notation, we continue calling $\rho_n$. Now $(\rho_n)_n$ is a nonincreasing sequence which converges pointwise to $1_{\{x_0\}}$, so $(\lan \mu,\rho_n \ran)_n$ converges to $\mu(\{x_0\})$.

On the other hand $|\lan \mu,\rho_n \ran|\le \|\mu\|_{H^{-1}_x}\|\rho_n\|_{H^1_x}$. For the $H^1_x$ norm of $\rho_n$, we have
\begin{align*}
\|\nabla \rho_n\|_{L^2_x}^2  = \int_{\mr^2} n^2 |\nabla \rho(nx)|^2 \rmd x = \int_{B_{1/n}(0)} |\nabla \rho(y)|^2 \rmd y
\end{align*}
and so $\|\nabla \rho_n\|_{L^2}$ converges to $0$ as $n\to\infty$. In a similar and easier way one sees that $\|\rho_n\|_{L^2}$ converges to $0$. So $\|\rho_n\|_{H^1}$ tends to $0$. Hence $(\lan \mu,\rho_n \ran)_n$ tends also to $0$ and therefore $\mu(\{x_0\})=0$.
\end{proof}

\begin{remark}
It is only for the previous Lemma that we need to use that the process $\xi$ is $H^{-1}_x$-valued.
\end{remark}

We are now able to conclude the convergence of the nonlinear term in \eqref{eq:stochEulervort}. Fix $\omega$ in a full measure set such that $(\td{\xi}^j)_j$ converges to $\td{\xi}$ in $C_t(\mc{M}_{x,M},w^\ast )$. Since $\td{\xi}$ belongs to $L^2_t(H^{-1}_x)$, $\td{\xi}_r$ belongs to $H^{-1}_x$ for all $r$ in a full measure set $S$ of $[0,T]$. In particular, by Lemma \ref{lem:no_atoms}, $\td{\xi}_r$ has no atoms. Hence, for all $r$ in $S$, Lemma \ref{lem:continuity_nonlin} implies the convergence of
\begin{align*}
\lan N(\xi^j),\varphi \ran = \lan \xi^j\otimes\xi^j, F_\varphi\ran
\end{align*}
towards the same term without $j$. By the dominated convergence theorem (in $r$), its time integral converges as well. This proves convergence for the nonlinear term.

\subsubsection{The stochastic integral}

It remains to prove convergence of the stochastic term. We follow the strategy in \cite{BrzGolJeg2013}. We use the notation $t^l_i = 2^{-l}i$, and define, for $t\in [0,T]$,
\begin{align*}
&Y_{j,k}(t) = \lan \td{\xi}^j_t, \sigma_k\cdot\nabla \varphi_t \ran,\\
&Y_{j,k}^{K,l}(t) = 1_{\{k\le K\}}\sum_i Y_{j,k}(t^l_i)1_{[t^l_i,t^l_{i+1}[}(t)
\end{align*}
and similarly without $j$. Finally we call $\rho_{j,k}$ the modulus of continuity of $Y_{j,k}$, namely
\begin{align*}
\rho_{j,k}(a) = \sup \{  |Y_{j,k}(t)-Y_{j,k}(s)|: t,s, \in [-,T] \mbox{ and } |t-s|\le a\}
\end{align*}
and similarly without $j$. Note that $\rho_{j,k}(a)$ and $\rho_k(a)$ are $\tilde{\mc{F}}_{T}$-measurable on $\tilde{\Omega}$, since the above supremum can be restricted to rational times $t,s$. We split
\begin{align*}
&\left|\sum_k \int^t_0 \lan \td{\xi}^j, \sigma_k\cdot\nabla \varphi \ran \rmd \td{W}^{(j),k} - \int^t_0 \lan \td{\xi}, \sigma_k\cdot\nabla \varphi \ran \rmd \td{W}^{k} \right|\\
&\le \left| \sum_k \int^t_0 (Y_{j,k}-Y^{K,l}_{j,k}) \rmd \td{W}^{(j),k} \right|\\
&\ \ \ + \left| \sum_k \int^t_0 Y^{K,l}_{j,k} \rmd \td{W}^{(j),k} - \int^t_0 Y^{K,l}_{k} \rmd \td{W}^{k} \right|\\
&\ \ \ + \left| \sum_k \int^t_0 (Y_{k}-Y^{K,l}_{k}) \rmd \td{W}^{k} \right|\\
&=:T_1+T_{2}+T_{3}.
\end{align*}
Concerning the first term  $T_1$, we have
\begin{align*}
\E\vert  T_1 \vert^2 = \sum_k \E \int^t_0 |Y_{j,k}-Y^{K,l}_{j,k}|^2 \rmd r.
\end{align*}
In order to have uniform estimates with respect to $j$, we want to use the convergence in $C_t$ of $Y^{j,k}$ to $Y^k$. For this, we split again the right-hand side:
\begin{align*}
&\sum_k \E \int^t_0 |Y_{j,k}-Y^{K,l}_{j,k}|^2 \rmd r\\
&\le C\sum_k \E \int^t_0 |Y^{K,l}_{j,k}-Y^{K,l}_{k}|^2 \rmd r + C \sum_k \E \int^t_0 |Y_{k}-Y^{K,l}_{k}|^2 \rmd r + C \sum_k \E \int^t_0 |Y_{j,k}-Y_{k}|^2 \rmd r\\
&=:T_{11}+T_{12}+T_{13}.
\end{align*}
For $T_{11}$, we have
\begin{align*}
&T_{11}= C \sum_{k\le K} \E \int^t_0 |Y^{K,l}_{j,k}-Y^{K,l}_{k}|^2 \rmd r\\
&\le C \sum_{k\le K} \E \sup_r |Y_{j,k}(r)-Y_{k}(r)|^2
\end{align*}
For $T_{13}$, we have similarly
\begin{align*}
&T_{13}= C \sum_{k\le K} \E \int^t_0 |Y_{j,k}-Y_{k}|^2 \rmd r + C \sum_{k>K} \E \int^t_0 |Y_{j,k}-Y_{k}|^2 \rmd r\\
&\le C \sum_{k\le K} \E \sup_r |Y_{j,k}(r)-Y_{k}(r)|^2 + C \sum_{k>K} \|\sigma_k\|_{C_x}^2
\end{align*}
where we have used that $\sup_r|Y_{j,k}(r)|\le C\|\sigma_k\|_{C_x}^2$. The constant $C$ here  depends on $M$, the upper bound of $\|\xi^j\|_{\mc{M}_x}$ and $\varphi$. For $T_{12}$, we have
\begin{align*}
&T_{12} = C \sum_{k\le K} \E \int^t_0 |Y_{k}-Y^{K,l}_{k}|^2 \rmd r + C \sum_{k>K} \E \int^t_0 |Y_{k}|^2 \rmd r\\
&\le C \sum_{k\le K} \E \rho_{k}(2^{-l})^2 + C \sum_{k>K} \|\sigma_k\|_{C_x}^2.
\end{align*}
This complete the bound for $T1$. Concerning the term $T_{3}$, we have
\begin{align*}
\E \vert T_{3}\vert^2 = \sum_k \E \int^t_0 |Y_{k}-Y^{K,l}_{k}|^2 \rmd r,
\end{align*}
which is the term $T_{12}$ up to a multiplicative constant and can therefore bounded as $T_{12}$. Finally, we note that the term $T_{2}$ can be written as
\begin{align*}
T_{2} = \left| \sum_{k\le K} \sum_i [Y_{j,k}(t^l_i) (W^{(j),k}_{t^l_{i+1}}-W^{(j),k}_{t^l_i}) - Y_{k}(t^l_i) (W^{k}_{t^l_{i+1}}-W^{k}_{t^l_i})] \right|
\end{align*}
Putting all together, we find
\begin{align*}
&\E \left|\sum_k \int^t_0 \lan \td{\xi}^j, \sigma_k\cdot\nabla \varphi \ran \rmd \td{W}^{(j),k} - \int^t_0 \lan \td{\xi}, \sigma_k\cdot\nabla \varphi \ran \rmd \td{W}^{k} \right|^2\\
&\le C \sum_{k\le K} \E \rho_{k}(2^{-l})^2 + C \sum_{k>K} \|\sigma_k\|_{C_x}^2
 + C \sum_{k\le K} \E \sup_r |Y_{j,k}(r)-Y_{k}(r)|^2\\
&\ \ \ + C\E \left| \sum_{k\le K} \sum_i [Y_{j,k}(t^l_i) (W^{(j),k}_{t^l_{i+1}}-W^{(j),k}_{t^l_i}) - Y_{k}(t^l_i) (W^{k}_{t^l_{i+1}}-W^{k}_{t^l_i})] \right|^2.
\end{align*}
We first choose a natural number  $K$ such that $\sum_{k>K} \|\sigma_k\|_{C_x}^2<\eps$. For each $k$, since $Y_k$ is a continuous function, $\rho_{k}(2^{-l})$ converges to $0$ as $l\to\infty$ $\td{\mathbb{P}}$-a.s.. Moreover $Y_k$ is also essentially bounded, therefore, by dominated convergence theorem, $\E \rho_{k}(2^{-l})^2$ also converges to $0$. Hence, for $K$ fixed as before, we can choose $l$ such that
\begin{align*}
\sum_{k\le K}\E \rho_{k}(2^{-l})^2<\eps.
\end{align*}
Again for each $k$, due to the convergence of $\td{\xi}^j$ in $C_t(\mc{M}_{x,M},w^\ast )$, we infer that $\sup_r |Y_{j,k}(r)-Y_{k}(r)|^2$ converges to $0$ as $j\to \infty$ $\td{\mathbb{P}}$-a.s.. Moreover $Y_{j,k}$ are bounded uniformly in $j$, therefore, by dominated convergence theorem, $\E \sup_r |Y_{j,k}(r)-Y_{k}(r)|^2$ also converges to $0$. Hence, for $K$ fixed as before, we can choose $\bar{j}\in\mathbb{N}$ such that, for every $j\ge \bar{j}$,
\begin{align*}
\sum_{k\le K}\E \sup_r |Y_{j,k}(r)-Y_{k}(r)|^2<\eps.
\end{align*}
Finally, for $K$, $l$ fixed as before, the term
\begin{align*}
\sum_{k\le K} \sum_i [Y_{j,k}(t^l_i) (W^{(j),k}_{t^l_{i+1}}-W^{(j),k}_{t^l_i}) - Y_{k}(t^l_i) (W^{k}_{t^l_{i+1}}-W^{k}_{t^l_i})]
\end{align*}
converges to $0$ as $j\to\infty$ $\td{\mathbb{P}}$-a.s.. Therefore, by the Lebesgue Dominated Convergence Theorem, also its second moment converges to $0$. Hence, for $K$, $l$ fixed as before, we can choose  $\bar{j}\in\mathbb{N}$ such that, for every $j\ge \bar{j}$,
\begin{align*}
\E \left| \sum_{k\le K} \sum_i [Y_{j,k}(t^l_i) (W^{(j),k}_{t^l_{i+1}}-W^{(j),k}_{t^l_i}) - Y_{k}(t^l_i) (W^{k}_{t^l_{i+1}}-W^{k}_{t^l_i})] \right|^2 <\eps.
\end{align*}
This proves that the stochastic term in \eqref{eq:stochEulervort} converges in $L^2_\omega$ norm, and so $\td{\mathbb{P}}$-a.s.~up to subsequences.

We have proved that all the terms in \eqref{eq:stochEulervort_test} passes to the $\td{\mathbb{P}}$-a.s.~limit, up to subsequences, and therefore $\td{\xi}$ is a solution to \eqref{eq:stochEulervort_test}, so to \eqref{eqn: Euler stoch vorticity form-Intro} with $\td{W}$ as a Brownian Motion. The proof of Lemma \ref{lem:limit_sol} is thus complete.
\end{proof}

\begin{appendices}

\section{The nonlinear term in the Euler Equations}
\label{app-nonlinear term}

\begin{proof}[Proof of Lemma \ref{lem:Poupaud_trick}]
A proof of this  Lemma is essentially due to Schochet \cite[Lemma 3.2 and discussion thereafter]{Sch1995}. We use here a version due to  Poupaud \cite[Section 2]{Pou2002}.
\begin{itemize}
\item Since $K$ is a smooth function outside the diagonal $\{x=y\}$, also the function $F_\varphi$ is smooth outside the diagonal. Recall that by Lemma \ref{lem:Green_function} we have $|K(x-y)|\le C|x-y|^{-1}$,  $(x,y)\in \mt^2$. Therefore, since  $\nabla \varphi$ is Lipschitz, we infer that
\begin{align*}
|F_\varphi(x,y)|\le \frac12 |K(x-y)|\|D^2\varphi\|_{C_x} |x-y|\le C\|\varphi\|_{C^2_x},
\end{align*}
which gives the desired bound on $F_\varphi$. The zero-average property \eqref{eqn-F average} of $F_\varphi$ is a consequence of the fact that $K$ is divergence-free.
\item For every $\xi \in \mc{M}_x$, for every $\varphi \in C^2_x$, $F_\varphi$ is bounded and so $\lan N(\xi),\varphi\ran$ is well-defined and
\begin{align*}
&|\lan N(\xi),\varphi\ran| = \left|\int\int F_\varphi(x,y) \xi(\rmd x)\xi(\rmd y)\right|\\
&\le C\|\varphi\|_{C^2_x} \|\xi^{\otimes 2}\|_{\mc{M}_{x,y}} \le C\|\varphi\|_{H^4_x} \|\xi\|_{\mc{M}_x}^2,
\end{align*}
where we used the Sobolev embedding in the last inequality. In particular $N(\xi)$ is a well-defined linear bounded functional on $H^4$.
\item By Lemma \ref{rmk:cont_prod_meas}  the map $\mc{M}_x\ni\xi\mapsto \xi\otimes \xi\in \mc{M}_{x,y}$ is Borel (with respect to the Borel $\sigma$-algebras generated by the weak-$^\ast$ topologies on $\mc{M}_x$ and $\mc{M}_{x,y}$). Moreover, by Lemma \ref{rmk:bounded_test_Borel},   for every $\varphi \in C^2_x$, the map
\begin{align*}
\mc{M}_{x,y}\ni \mu\mapsto \lan \mu, F_\varphi \ran \in \mr
\end{align*}
is Borel. Therefore, for every $\varphi \in C^2_x$, the function
\begin{align*}
\mc{M}_x\ni\xi\mapsto \lan N(\xi),\varphi\ran \in \mr
\end{align*}
is Borel. Hence we deduce that  $N: \mc{M}_x \to  H^{-4}_x$ is weakly-$\ast$ Borel. Since $H^{-4}_x$ is a separable reflexive Banach space, by Lemma \ref{lem:equiv_topol} we infer that $N$ is also Borel.
\item Recall that $K$ is odd by Lemma \ref{lem:Green_function}, therefore with $u=K\ast \xi$ we have
\begin{align*}
&\int \xi(x) u(x)\cdot \nabla \varphi(x) \rmd x = \int\int \xi(x)\xi(y) K(x-y) \cdot\nabla\varphi(x) \rmd x\rmd y\\
&= - \int\int \xi(x)\xi(y) K(y-x) \cdot\nabla\varphi(x) \rmd x\rmd y\\
&= - \int\int \xi(x)\xi(y) K(x-y) \cdot\nabla\varphi(y) \rmd x\rmd y,
\end{align*}
where the  last equality follows by  swapping  $x$ with $y$. Hence
\begin{align*}
\int \xi(x) u(x)\cdot \nabla \varphi(x) \rmd x &=\int \xi(x) u(x)\cdot \nabla \varphi(x) \rmd x\\
&\hspace{-1truecm}= \frac12 \int\int \xi(x)\xi(y) K(x-y) \cdot (\nabla_x\varphi(x) -\nabla_y\varphi(y)) \rmd x\rmd y.
\end{align*}
\item Continuity of $N$ on $\mc{M}_{x,M,+,\mathrm{no-atom}}$ follows from Lemma \ref{lem:continuity_mass}.
\item The fact that $N(\xi+\alpha)=N(\xi)$ follows from the zero-average property \eqref{eqn-F average} of $F_\varphi$.
\end{itemize}
\end{proof}

\section{Some technical lemmata}
\label{app-B}

\begin{lemma}\label{lem:H_Borel}
Assume that the vector fields $\sigma_k$, $k\in \mathbb{N}$,  satisfy Assumption \ref{assumption_sigma}.  Then the following three  maps
\begin{align}
\begin{aligned}\label{eq:maps_M_H}
&\mc{M}_x \ni \mu \mapsto \mu \in H^{-4}_x,\\
&\mc{M}_x \ni \mu \mapsto \sigma_k\cdot\nabla \mu \in H^{-4}_x,\quad k\in\mathbb{N},\\
&\mc{M}_x \ni \mu \mapsto \Delta \mu \in H^{-4}_x
\end{aligned}
\end{align}
are well defined and  linear and continuous in both the norm-to-norm and the weak-$^\ast$-to-weak sense and there exists  a number $C>0$ such that for every $\mu \in\mc{M}_x$,
\begin{align}\label{eqn-B2}
\|\mu\|_{H^{-4}_x} +\|\Delta \mu\|_{H^{-4}_x} \le C\|\mu\|_{\mc{M}_x},\\
\|\sigma_k\cdot\nabla \mu\|_{H^{-4}_x}\le C\|\sigma_k\|_{C_x} \|\mu\|_{\mc{M}_x}.
\end{align}
Moreover,  if  $\mc{M}_x$ is endowed with the weak-$^\ast$ topology,  these three maps are  Borel.
\end{lemma}

\begin{proof}
The maps in \eqref{eq:maps_M_H}, tested against a test function $\varphi$, read formally
\begin{align*}
&\mu \mapsto \lan \mu,\varphi \ran,\\
&\mu \mapsto -\lan \mu,\sigma_k\cdot\nabla \varphi\ran,\\
&\mu \mapsto \lan \mu, \Delta \varphi\ran.
\end{align*}
Now, by Sobolev embedding, for any $\varphi \in H^4_x$, the functions $\varphi$, $\sigma_k\cdot\nabla \varphi$, $\Delta\varphi$ are continuous with
\begin{align*}
&\|\varphi\|_{C_x} \le C\|\varphi\|_{H^4_x},\\
&\|\sigma_k\cdot\nabla \varphi\|_{C_x} \le \|\sigma_k\|_{C_x} \|\nabla\varphi\|_{C_x} \le C\|\sigma_k\|_{C_x} \|\nabla\varphi\|_{H^4_x},\\
&\|\Delta\varphi\|_{C_x} \le C\|\varphi\|_{H^4_x}.
\end{align*}
Hence the maps in \eqref{eq:maps_M_H} are weak-$^\ast$-to-weak continuous. Taking the supremum over $\varphi$ in the unit ball of $H^4_x$, we get also the norm-to-norm continuity and the desired bounds.
\end{proof}

\begin{lemma}\label{lem:stochEulervort_H}
Assume Assumption \ref{assumption_sigma} on $\sigma_k$ and  fix $M>0$. Then for any object $(\Omega,\mc{A},\mathbb{F},\mathbb{P},W,\xi)$, where $\mathbb{F}=(\mc{F}_t)_t$ and $W=(W^k)_k$ is a filtration and, respectively an $\mathbb{F}$-cylindrical Wiener process on a filtered probability space $(\Omega,\mc{A}, \mathbb{F},\mathbb{P})$ satisfying  the usual assumptions,  and $\xi:[0,T]\times\Omega\to\mc{M}_{x,M}$ is $\mc{B}\mathbb{F}$ Borel measurable function, equality  \eqref{eq:stochEulervort} holds if and only if equality \eqref{eq:stochEulervort_test} holds for every $\varphi \in C^\infty_x$.
\end{lemma}

\begin{proof}
Assume that \eqref{eq:stochEulervort} holds, fix $\varphi \in C^\infty_x$. Then \eqref{eq:stochEulervort_test} follows by applying the linear continuous functional $\lan \cdot,\varphi\ran$ to \eqref{eq:stochEulervort} and exchanging the functional with the integrals.

Conversely, assume that \eqref{eq:stochEulervort_test} holds for every $\varphi \in C^\infty_x$. Lemma \ref{lem:H_Borel} implies that $\xi$ and all the integrands in \eqref{eq:stochEulervort} are progressively measurable as $H^{-4}_x$ processes and that the deterministic and stochastic integrals are well-defined (see \eqref{eq:welldef_stoch_int}). Now we have for every test function $\varphi \in C^\infty_x$, by \eqref{eq:stochEulervort_test} exchanging $\lan \cdot,\varphi\ran$ and the integrals,
\begin{align*}
\lan \xi_t,\varphi \ran &= \lan \xi_0,\varphi\ran +\lan \int^t_0 N(\xi_r) \rmd r, \varphi \ran\\
&\ \ \ -\lan \sum_k \int^t_0 \sigma_k\cdot\nabla \xi_r \ran \rmd W^k_r , \varphi \ran\\
&\ \ \ +\frac{c}2 \lan \int^t_0 \Delta \xi_r \rmd r, \varphi \ran, \quad\text{for every }t,\quad \mathbb{P}-\text{a.s.},
\end{align*}
that is \eqref{eq:stochEulervort} tested against $\varphi$, where the $\mathbb{P}$-exceptional set can depend on $\varphi$. Taking $\varphi$ in a countable set of $C^\infty_x$, dense in $H^4_x$, we deduce \eqref{eq:stochEulervort} in $H^{-4}_x$. The proof is complete.
\end{proof}


\begin{remark}\label{rmk:Linfty_sol}
Let $\xi^\eps_t(\omega)=(\Phi^\eps(t,\cdot,\omega))_\# \xi^\eps_0$ be defined as at the beginning of Section \ref{sec:apriori_bd}. For a.e.~$\omega$, for every $t$, $\xi^\eps_t \in \mc{M}_{x,M}$.  Moreover, the process $\lan \xi^\eps_t,\varphi\ran$ is progressively measurable and, because it satisfies \eqref{eq:stochEulervort_test},  continuous  for every $\varphi \in C^\infty_x$, and hence, by the density of $C^\infty_x$ in $C_x$, for every $\varphi \in C_x$. Hence, by Lemma \ref{rmk:Borel_weakstar} and Remark \ref{rmk:restriction}, up to redefining $\xi^\eps=0$ on the $\mathbb{P}$-exceptional set where $\xi^\eps_t$ is not in $\mc{M}_{x,M^\eps}$ for some $t$, $\xi^\eps$ is $\mc{B}\mathbb{F}$ Borel as $\mc{M}_x$-valued map and satisfies \eqref{eq:stochEulervort_test} for every $\varphi \in C^\infty_x$. By Lemma \ref{lem:stochEulervort_H}, $\xi^\eps$ is a $\mc{M}_{x,M}$-valued solution in the sense of Definition \ref{def:sol}.
\end{remark}

Before we can embark on the proof of Lemma \ref{lem:velocity_eq} we formulate the following axially result.

\begin{lemma}\label{lem-auxilairy identity}
  If the $C^2$-class vector fields  $v,w:\mt^2\to\mr^2$  are such that $\divv v=0$, then
\begin{align}
\curl [v\cdot\nabla w +(Dv)^{\mathrm{t}} w] = v\cdot\nabla (\curl  w).
\label{eqn-vector_equality curl}
\end{align}
Here $Dv= \begin{pmatrix}                                D_1v^1 & D_2v^1 \\
                                D_1 v^2 & D_2 v^2
                              \end{pmatrix}$
 and ${}^{\mathrm{t}}$ is the transpose operation.
\end{lemma}
 \begin{proof}[Proof of Lemma \ref{lem-auxilairy identity}]
 The expression $(Dv)^{\mathrm{t}} w$ is understood here as the $2\times 1$ matrix  that is equal to the  multiplication of a $2\times 1$ matrix $\wek{w^1}{w^2}$   by a $2\times 2$ matrix $\mat{D_1v^1}{ D_1v^2}{ D_2 v^1}{D_2 v^2}$, i.e.
\begin{align}
(Dv)^{\mathrm{t}} w:= \begin{pmatrix}D_1v^1 & D_1v^2 \\
                                D_2 v^1 & D_2 v^2
                              \end{pmatrix} \wek{w^1}{w^2}
                                            =\wek{w^1 D_1v^1   +  w^2 D_1v^2}{ w^1 D_2 v^1 + w^2 D_2 v^2 }
\label{eqn-B.05}
\end{align}
Therefore, since $\curl u:= D_1u^2-D_2u^1$ and $D_1D_2-D_2D_1=0$ we get
\begin{align}\nonumber
\curl[ (Dv)^{\mathrm{t}} w]&= D_1( w^1 D_2 v^1 + w^2 D_2 v^2 ) -D_2( w^1 D_1v^1   +  w^2 D_1v^2)
\\&= \bigl( D_1D_2 v^1-   D_2D_1v^1 \bigr)w^1+ \bigl( D_1D_2 v^2- D_2D_1v^2 \bigr)w^2
\nonumber
\\&+ D_1 w^1 D_2 v^1 + D_1 w^2 D_2 v^2  -D_2 w^1 D_1v^1   -D_2 w^2 D_1v^2
\nonumber
\\&= D_1 w^1 D_2 v^1 + D_1  w^2 D_2 v^2  -D_2  w^1 D_1v^1    -D_2    w^2 D_1v^2.
\label{eqn-B.06}
\end{align}
Now we will prove another auxiliary identity:
\begin{align}
\curl[ v\cdot \nabla w]=  v\cdot \nabla \curl[ w]+ \sum_{i} (D_1 v^iD_iw^2  -D_2 v^iD_iw^1 )
\label{eqn-B.07}
\end{align}
In the above formula, by $v\cdot \nabla w$ we mean the following vector field $\sum_{i} v^iD_iw$ and for $\mathbb{R}$-valued function $f$, we put
$v\cdot \nabla f:=\sum_{i} v^iD_if$.

Let us now prove identity  \eqref{eqn-B.07}.
\begin{align*}
\curl[ v\cdot \nabla w]&=
\curl\bigl[ \sum_{i} v^iD_iw\bigr]= D_1\bigl[ \sum_{i} v^iD_iw^2 \bigr] -D_2\bigl[ \sum_{i} v^iD_iw^1 \bigr]
\\&=  \sum_{i} v^i D_1 D_iw^2  + \sum_{i} D_1 v^iD_iw^2
\\&-  \sum_{i} v^i D_2 D_iw^1  - \sum_{i} D_2 v^iD_iw^1
 \\&=  \sum_{i} v^i ( D_1 D_iw^2 -D_2 D_iw^1)+ \sum_{i} (D_1 v^iD_iw^2  -D_2 v^iD_iw^1 )
\\&=  \sum_{i} v^i D_i( D_1 w^2 -D_2 w^1)+ \sum_{i} (D_1 v^iD_iw^2  -D_2 v^iD_iw^1 )
\\&=  \sum_{i} v^i D_i( \curl w)+ \sum_{i} (D_1 v^iD_iw^2  -D_2 v^iD_iw^1 ).
\end{align*}

 By \eqref{eqn-B.06} and \eqref{eqn-B.07} we have
\begin{align}
\label{eqn-B.08}
&\hspace{-1truecm}\curl [v\cdot\nabla w +(\nabla v)^{\mathrm{t}}\cdot w] - v\cdot\nabla \curl [w]
\\&=  D_1 v^1D_1w^2   -D_2 v^1D_1w^1  +  D_1 v^2D_2w^2  -D_2 v^2D_2w^1
\nonumber\\&+  D_1 w^1 D_2 v^1 + D_1  w^2 D_2 v^2  -D_2  w^1 D_1v^1    -D_2    w^2 D_1v^2 .
\nonumber\\&=  D_1 v^1D_1w^2 -D_2 v^2D_2w^1 + D_1  w^2 D_2 v^2  -D_2  w^1 D_1v^1
\nonumber\\&= D_1 v^1 (D_1w^2  -D_2  w^1)+ D_2 v^2( D_1  w^2 -D_2w^1)
\nonumber\\&= \divv [v]\curl [w].
\nonumber\end{align}
Note that the above identity is valid for all $C^2$-class vector fields. But if additionally $\divv v=0$, it implies the identity
 \eqref{eqn-vector_equality curl}. The prof is complete.
 \end{proof}

\begin{lemma}\label{lem-main identity}
Assume that $p \in (2,\infty)$. Then for all  $\xi \in L^p_x$ and  $v \in \nH^{1,p}_x \cap \rH$, the following equality holds in  $H^{-1}_x$:
 \begin{align}
\Pi[v\cdot\nabla K \ast \xi +(Dv)^{\mathrm{t}} K \ast \xi] = K*[v\cdot\nabla \xi].
\label{eq:vel_vort}
\end{align}
\end{lemma}
\begin{proof}[Proof of Lemma \ref{lem-main identity}] Let us choose and  fix $p \in (2,\infty)$.
 We begin by observing  that by part 5 of Lemma \ref{lem:Green_function}, identity \eqref{eqn-vector_equality curl} implies that
\begin{align*}
\Pi[v\cdot\nabla w +(D v)^{\mathrm{t}} w] = K*[v\cdot\nabla \curl w],
\end{align*}
where $\Pi$ the Leray-Helmholtz  projection.
 Take $w=K \ast \xi$, for a regular scalar function $\xi:\mt^2\to\mr$. By  Lemma \ref{lem:Green_function}, $\curl [w] = \xi -\gamma$, where $\gamma = \int_{\mt^2} \xi \in \mathbb{R}$.
 If we take  a regular and divergence-free vector field  $v:\mt^2\to\mr^2$,  then by applying  the above formula  we infer that
 \eqref{eq:vel_vort}  holds.
 We claim that   equality \eqref{eq:vel_vort}  in the $H^{-1}_x$ sense  is satisfied for every  $\xi \in L^p_x$ and  $v \in \nH^{1,p}_x \cap \rH$.
 Since
   by  Lemma \ref{lem:Green_function}, both sides of \eqref{eq:vel_vort} are continuous bilinear maps from $ L^p_x \times \nH^{1,p}_x \cap \rH$ to $H^{-1}_x$, the result follows.
\end{proof}

After having proved the last Lemma we are ready to embark on our main task.

\begin{proof}[Proof of Lemma \ref{lem:velocity_eq}] Let us choose and  fix $p \in (2,\infty)$.
Assume  that a process    $\xi$ belongs to $L^p_{t,\omega}(L^p_x)$  and that
  $\xi$ is  an $\mc{M}_{x,M}$-valued distributional solution to the stochastic vorticity equation \eqref{eqn: Euler stoch vorticity form-Intro} (according to Definition \ref{def:sol}).
  Thus equality  \eqref{eq:stochEulervort} is satisfied in $H_x^{-4}$.
   Define a process $u$ by
  \[ u(t,\cdot,\omega) =K \ast \xi (t,\cdot,\omega), \;\; (t,\omega) \in [0,\infty)\times \Omega.\]

Since $\xi$ belongs to  $L^p_{t,\omega}(L^p_x)$, by by Lemma \ref{lem:Green_function} we infer that  $u$ belongs  $L^p_{t,\omega}(\nH^{1,p}_x)$. We  also observe by Lemma \ref{lem:Poupaud_trick},  that since  $\xi \in L^p_{t,\omega}(L^p_x)$, the nonlinear term can be written as $u\cdot\nabla\xi$, and the equation \eqref{eq:stochEulervort} holds actually in $H^{-2}_x$.  Indeed, $\xi$ and all the integrands of \eqref{eq:stochEulervort} takes values in $H^{-2}_x$ and are progressively measurable as $H^{-2}_x$-valued processes (and their deterministic and stochastic $H^{-2}_x$-valued integrals coincide with the $H^{-4}_x$-valued integrals). Now we apply to \eqref{eq:stochEulervort} the convolution operator
\begin{align*}
K\ast: H^{-2}_x \ni \xi \mapsto u=: K\ast \xi \in  H^{-1}_x,
\end{align*}
which by Lemma \ref{lem:Green_function} is linear and bounded. By the first part of this proof, we get, for a.e.~$\omega$, as equality in $H^{-1}_x$: for every $t$,
\begin{align*}
u_t &= u_0 - \int^t_0 \Pi[u_r\cdot\nabla u_r +(Du_r)^{\mathrm{t}} u_r] \rmd r\\
&- \sum_k \int^t_0 \Pi[ \sigma_k\cdot\nabla u_r +(D\sigma_k)^{\mathrm{t}} u_r ] \rmd W^k_r\\
&+\frac{c}2 \int^t_0 \Delta u_r \rmd r
\end{align*}
Now we note that $(Du_r)^{\mathrm{t}} u_r = \nabla [|u_r|^2]/2$ and hence $\pi\bigl( (Du_r)^{\mathrm{t}} u_r\bigr)=0$. Thus we deduce  \eqref{eqn-stochEulervel} and so
the proof of Lemma \ref{lem:velocity_eq} is complete.
\end{proof}

\begin{proof}[Proof of  Lemma \ref{lem:u_norm}]
Leu us fix $p\in (2,\infty)$ and the processes $\xi$ and $u$ as in the proof  of Lemma \ref{lem:velocity_eq}.

For any $\delta \in (0,1]$, we consider  an operator
\begin{align*}
R^\delta:  H^{-1}_x u \mapsto \rho_\delta \ast u \in  L^2_x
\end{align*}
where $(\rho_\delta)_\delta$ is a standard family of mollifiers on $\mt^2$.
Let us observe that the operators  $R^\delta$  are uniformly  bounded maps  on the spaces $L^p_x$ and $\nH^{1,p}_x$. In particular, the processes   $R^\delta u$  belong uniformly to the same spaces as the process $u$. Hence, by the Sobolev embedding, the $L^p_{t,\omega}(L^\infty_x)$ norm of  $R^\delta u$ is uniformly bounded.


Note that $R^\delta f \to f$ in $L^2_x$, resp.~$H^{-1}_x$ for every $f \in L^2_x$, resp.~$f \in H^{-1}_x$. By applying the  map  $R^\delta$ to equality \eqref{eqn-stochEulervel} we get
\begin{align*}
&R^\delta u_t = R^\delta u_0 -\int^t_0 R^\delta \Pi[u_r\cdot\nabla u_r] dr\\
&-\int^t_0 R^\delta \Pi[ \sigma_k\cdot\nabla u_r +(D\sigma_k)^{\mathrm{t}} u_r ] \rmd W^k_r\\
&+\frac{c}2 \int^t_0 R^\delta\Delta u_r \rmd r, \;\;\mbox{ in } L^2_x
\end{align*}
 for every $t \geq 0$. Next we can apply the It\^o formula  \cite[Theorem 4.32]{DaPZab2014} to the square of the $L^2_x$ norm, which is obviously of $C^2$ class on the space $L^2_x$,  with uniformly continuous derivatives on bounded subsets of $L^2_x$. We infer that, for every $t$,
\begin{align*}
&\|R^\delta u_t\|_{L^2_x}^2 = \|R^\delta u_0\|_{L^2_x}^2 -2\int^t_0 \lan R^\delta u_r, R^\delta \Pi[u_r\cdot\nabla u_r] \ran dr\\
&-2\int^t_0 \lan R^\delta u_r, R^\delta \Pi[ \sigma_k\cdot\nabla u_r +(D\sigma_k)^{\mathrm{t}} u_r ] \ran \rmd W^k_r\\
&+\int^t_0 \lan R^\delta, cR^\delta\Delta u_r \ran \rmd r + \int^t_0 \sum_k \|R^\delta \Pi[ \sigma_k\cdot\nabla u_r +(D\sigma_k)^{\mathrm{t}} u_r\|_{L^2_x}^2 \rmd r.
\end{align*}
Since
\begin{align*}
&\sum_k \E \int^T_0 |\lan R^\delta u_r, R^\delta \Pi[ \sigma_k\cdot\nabla u_r +(D\sigma_k)^{\mathrm{t}} u_r ] \ran|^2_{L^2_x} \rmd r\\
&\le C \|u\|_{L^2_{t,\omega}(L^\infty_x)}^2 \left(\sum_k \|\sigma_k\|_{C^1_x}\right) \|u\|_{L^2_{t,\omega}(\nH^{1,2}_x)}^2.
\end{align*}
we infer that  the stochastic integral is an $L_x^2$-martingale with zero mean. Similarly the integrands in the deterministic integrals have finite $L^1_{t,\omega}$ norm and we can take expectation: we get
\begin{align*}
&\E\|R^\delta u_t\|_{L^2_x}^2 = \E\|R^\delta u_0\|_{L^2_x}^2 -2\E \int^t_0 \lan R^\delta u_r, R^\delta \Pi[u_r\cdot\nabla u_r] \ran dr\\
&-\E\int^t_0 c\E \|R^\delta \nabla u_r\|_{L^2_x}^2 \rmd r + \E \int^t_0 \sum_k \|R^\delta \Pi[ \sigma_k\cdot\nabla u_r +(D\sigma_k)^{\mathrm{t}} u_r\|_{L^2_x}^2 \rmd r,
\end{align*}
where we have used integration by parts and that $R^\delta$ commutes with $\Delta$ and $\nabla$. Finally, we note that $R^\delta f \to f$ in $L^2_x$ for every $f$ in $L^2_x$. We exploit this fact for $f=u_r$, $f=\Pi[u_r\cdot\nabla u_r]$, $f=\sigma_k\cdot\nabla u_r +(D\sigma_k)^{\mathrm{t}} u_r$ and $f=\nabla u_r$, and use the dominated convergence theorem in $r$ and $\omega$ and $k$, to pass $\delta\to 0$ and obtain \eqref{eq:u_norm}. The proof is complete.
\end{proof}

\begin{lemma}\label{lem:C_Lebesgue_meas}
Let $X$ be a closed convex subset of a topological vector space, endowed with its Borel $\sigma$-algebra. Assume that $X$ is also a Polish space. Let $\zeta:[0,T]\times \Omega\to X$ be a $\mc{B}([0,T])\times \mc{A}$ Borel measurable map.
\begin{itemize}
\item The set $C_t(X)$ is a Polish space and, if, for every $\omega$, $t\mapsto \zeta_t$ is in $C_t(X)$, then $\omega\mapsto \zeta(\cdot,\omega)$ is $\mc{A}$ Borel measurable as $C_t(X)$-valued map.
\item If $X$ is a separable reflexive Banach space and, for every $\omega$, $t\mapsto \zeta_t$ is in $L^2_t(X)$ (more precisely, has finite $L^2_t(X)$ norm), then $\omega\mapsto \zeta(\cdot,\omega)$ (more precisely, its equivalence class) is $\mc{A}$ Borel measurable as $L^2_t(X)$-valued map.
\end{itemize}
\end{lemma}

\begin{proof}
For the first point, the fact that $C_t(X)$ is a Polish space is well known. Moreover the Borel $\sigma$-algebra $\mc{B}(C_t(X))$ on $C_t(X)$ is generated by the evaluation maps $\pi_t(\gamma)=\gamma_t$. Indeed,  $\mc{B}(C_t(X))$ is generated by the maps
\begin{align*}
\gamma \mapsto d(\gamma(t),g(t)),\quad t\in [0,T]\cap \mathbb{Q},\quad g\in C_t(X),
\end{align*}
with $d$ distance on $X$, and these maps are measurable in the $\sigma$-algebra generated by the evaluation maps (because they are composition of the evaluation maps and a Borel function on $X$). Now, for every $t$, the map $\pi_t(\zeta)=\zeta_t$ is $\mc{A}$ Borel, by the Fubini theorem, hence, if $\zeta$ is $C_t(X)$-valued, then it is $\mc{A}$ Borel measurable as $C_t(X)$-valued map.

For the second point, we note that, by Lemma \ref{lem:equiv_topol}, it is enough to show that $\zeta$ is weakly progressively measurable. Since the dual of $L^2_t(X)$ is $L^2_t(X^\ast)$ (see \cite[Chapter IV Section 1]{DieUhl1977}), it is enough to show that, for every $\varphi \in L^2_t(X^\ast)$,
\begin{align*}
\omega\mapsto \int^T_0 \lan\zeta(t,\omega), \varphi(t)\ran_{X,X^\ast} \rmd t
\end{align*}
is measurable. But this follows from Fubini theorem. The proof is complete.
\end{proof}

\begin{proof}[Proof of Lemma \ref{lem:BM_enlarged}]
Let us denote by $\td{\mathbb{F}}^{00}=\big(\td{\mc{F}}^{00}_t\bigr)_t$, where $\td{\mc{F}}^{00}_t = \sigma\{\td{\xi}_s,\td{W}_s : 0\le s\le t\}$,   the filtration generated by $\td{\xi}$ and $\td{W}$.
Clearly the processes $\td{W}$ and $\td{\xi}$ are adapted to $\td{\mc{F}}^{00}$. We claim that $\td{W}$ is a cylindrical $\td{\mc{F}}^{00}$-Wiener process. Indeed, $\td{W}$ is a cylindrical Brownian motion with respect to its natural filtration, as a.s. limit of cylindrical Brownian motions. Moreover, for every $0\le s_1\le \ldots \le s_h\le s < t$, $\td{W}^{(j)}_t-\td{W}^{(j)}_s$ is independent of $(\td{\xi}^j_{s_1},\td{W}^{(j)}_{s_1},\ldots \td{\xi}^j_{s_h},\td{W}^{(j)}_{s_h})$, therefore $\td{W}_t-\td{W}_s$ is independent of $(\td{\xi}_{s_1},\td{W}_{s_1},\ldots \td{\xi}_{s_h},\td{W}_{s_h})$. This proves our claim.

Recall that $\td{\mathbb{F}}^{0}=\big(\td{\mc{F}}_t^{0}\bigr)_t$ is the filtration generated by $\td{\mathbb{F}}^{00}$ and the $\td{\mathbb{P}}$-null sets on $(\td{\Omega},\td{\mc{A}},\td{\mathbb{P}})$ and that $\td{\mc{F}}_t = \cap_{s>t}\td{\mc{F}}^{0}_s$. We argue as in the proof of \cite[Proposition 2.5, Point 1]{Bas2011} (note that the proof is valid for any filtration making $W^k$ Brownian motions) and we get that the filtration $(\td{\mc{F}}_t)_t$ is complete and right-continuous and $\td{W}$ is still a cylindrical Brownian motion with respect to it. Finally $\td{\xi}$ is an $(\mc{M}_{x,M},w^\ast )$-valued $(\td{\mc{F}}_t)_t$-adapted and continuous process, hence also progressively measurable.

In a similar (and easier) way, one gets the result for $(\td{\mc{F}}^j_t)_t$, $\td{W}^{(j)}$ and $\td{\xi}^j$, for each $j$.
\end{proof}

\begin{proof}[Proof of Lemma \ref{lem:eq_enlarged}]
Let us fix $j$. We have to verify equation \eqref{eq:stochEulervort} for $(\xi^j,W^{(j)})$ for every $\varphi \in C^\infty_{t,x}$. The idea is taken by \cite[Section 5]{BrzGolJeg2013}: it is enough to verify that, for every $\varphi \in C^\infty_x$, for every $t$, the random variables
\begin{align}
Z_t&:= \lan \xi^{\eps_j}_t,\varphi_t \ran - \lan \xi^{\eps_j}_0,\varphi\ran -\int^t_0 \lan N(\xi^{\eps_j}_r),\varphi \ran \rmd r\nonumber\\
& -\sum_k \int^t_0 \lan \xi^{\eps_j}_r,\sigma_k\cdot\nabla \varphi \ran  \rmd W^k_r -\frac12 \int^t_0 \lan \xi^{\eps_j}_r, c\Delta \varphi \ran \rmd r \label{eq:vorticity_td_approx}
\end{align}
and $\td{Z}_t$, obtained as in \eqref{eq:vorticity_td_approx} replacing $(\xi^{\eps_j},W)$ with $(\td{\xi}^j,\td{W}^{(j)})$, have the same law. We fix $t$ and $\varphi \in C^\infty_x$. By Lemma \ref{lem:Poupaud_trick} and Lemma \ref{lem:H_Borel}, all the terms in \eqref{eq:vorticity_td_approx} but the nonlinear term and the stochastic integral are Borel functions of $\xi^{\eps_j}$ with respect to the $C_t(\mc{M}_{x,M},w^\ast )$ topology. Concerning the stochastic integral we use an approximation argument. For every positive integers $K$ and $N$, calling $t^N_i =2^{-N}i$ for $i$ integer, the map
\begin{align*}
C_t(\mc{M}_{x,M},w^\ast )\times C_t^{\mathbb{N}}\ni (\xi,W)\mapsto \sum_{k=1}^K \sum_{i,t^N_{i+1}\le t} \lan \xi_{t^N_i}, \sigma_k\cdot\nabla \varphi{t^N_i} \ran (W_{t^N_{i+1}}-W_{t^N_i})
\end{align*}
is a continuous, in particular Borel function. By the continuity of $t\mapsto \lan \xi_t, \sigma_k\cdot\nabla \varphi_t \ran$ for every $k$, for a.e.~$\omega$, and by the square-summability of $\|\sigma_k\|_{C_x}$, we get via the dominated convergence theorem that, as $(N,K)$ tends to $\infty$,
\begin{align*}
\sum_k \E \int^T_0 |\lan \xi_t, \sigma_k\cdot\nabla \varphi_t \ran - 1_{k\le K} \sum_i \lan \xi^{\eps_j}_{t^N_i}, \sigma_k\cdot\nabla \varphi{t^N_i} \ran 1_{[t^N_{i+1},t^N_i)}(t)|^2 \rmd r \to 0,
\end{align*}
so by the It\^o isometry we obtain that, as $(N,K)\to\infty$,
\begin{align*}
\sum_{k=1}^K \sum_{i,t^N_{i+1}\le t} \lan \xi^{\eps_j}_{t^N_i}, \sigma_k\cdot\nabla \varphi{t^N_i} \ran (W_{t^N_{i+1}}-W_{t^N_i}) \to \sum_k \int^t_0 \lan \xi^{\eps_j}_r,\sigma_k\cdot\nabla \varphi_r \ran  \rmd W^k_r \quad \text{in }L^2_\omega.
\end{align*}
Similarly for $(\td{\xi}^j,\td{W}^{(j)})$ (with convergence in $L^2_{\tilde{\omega}}$). We conclude that
\begin{align*}
&Z_t = F_t(\xi^{\eps_j}) +L^2_\omega-\lim_{N,K} G_{N,K,t}(\xi^{\eps_j},W),\\
&\tilde{Z}_t = F_t(\tilde{\xi}^j) +L^2_{\tilde{\omega}}-\lim_{N,K} G_{N,K,t}(\tilde{\xi}^j,\tilde{W}^{(j)})
\end{align*}
for some Borel maps $F_t$ and $G_{N,K,t}$. Since $(\xi^{\eps_j},W)$ and $(\td{\xi}^j,\td{W}^{(j)})$ have the same law, also $Z_t$ and $\td{Z}_t$ have the same law. Since, $\mathbb{P}$-a.s.,~$Z_t=0$ for every $t$, also, $\mathbb{P}$-a.s.,~$\td{Z}_t=0$ for every $t$, and so, by Lemma \ref{lem:stochEulervort_H}, $(\td{\Omega},\td{\mc{A}},(\td{\mc{F}}^j_t)_t,\td{\mathbb{P}},\td{W}^{(j)},\td{\xi}^j)$ solves \eqref{eqn: Euler stoch vorticity form-Intro}.

Concerning Lemmas \ref{lem-Hm1_bound} and \ref{lem:Hm4_bound}, for any integer $h$, as a consequence of Remark \ref{rmk:Borel_norm}, the maps
\begin{align*}
&C_t(\mc{M}_{x,M},w^\ast )\ni \xi \mapsto (\|\xi_t\|_{H^h_x})_t \in C_t,\\
&C_t(\mc{M}_{x,M},w^\ast )\ni \xi \mapsto \|\xi\|_{C^\alpha_t(H^h_x)} \in \mr
\end{align*}
are Borel. Hence $\|\xi^{\eps_j}_t\|_{H^h_x}$ and $\|\td{\xi}^j_t\|_{H^h_x}$ have the same laws (as $C_t$-valued random variables) and so Lemma \ref{lem-Hm1_bound} holds for $\td{\xi}^j$. Similarly $\|\xi^{\eps_j}_t\|_{H^h_x}$ and $\|\td{\xi}^j_t\|_{H^h_x}$ have the same laws and so Lemma \ref{lem:Hm4_bound} holds for $\td{\xi}^j$.

Finally, concerning non-negativity, we note that the set $\{\xi_t\ge 0,\, \forall t\}$ is Borel in $C_t(\mc{M}_x,w^\ast )$, because it can be written as $\lan \xi_t,\varphi \ran\ge 0$ for all rational $t$ and all $\varphi$ in a countable dense set in $C_x$. Since $\xi^{\eps_j}$ is concentrated on $\{\xi_t\ge 0,\, \forall t\}$, also $\td{\xi}^j$ is concentrated on this set. The proof is complete.
\end{proof}

\section{The torus and the corresponding  Green function}
\label{app-C}

We consider the torus $\mt^2$ as the two-dimensional manifold obtained from $[-1,1]^2$ identifying the opposite sides; we call $\pi: \mathbb{R}^2 \to \mathbb{T}^2$ the quotient map. A continuous ($C_x$) function is understood here as a continuous periodic function on $\mr^2$, with period $2$ on both $x_1$ and $x_2$ directions, and can be identified with a continuous function on the torus $\mt^2$. For $s$ positive integer, a $C^s_x$ function on $\mt^2$ is a $C^s$ periodic function on $\mr^2$ (with period $2$). Similarly, for $s$ positive integer and $1\le p \le\infty$, a $\nH^{s,p}_x$ function on $\mt^2$ is a $\nH^{s,p}_{loc}$ periodic function on $\mr^2$ (with period $2$). One can also define a Riemannian structure on the torus via the quotient map $\pi$ so that $\pi$ is a local isometry; the local isometry implies that the gradient, the covariant derivatives etc transform naturally, moreover the $C^s$ and $\nH^{s,p}$ spaces defined via the Riemannian structure coincide with the corresponding spaces of periodic functions as defined above.

The space of distribution $\mathcal{D}^\prime_x$ on $\mt^2$ is understood as the dual space of $C^\infty$ periodic functions on $\mr^2$. The spaces of functions can be identified with subspaces of distribution via the $L^2$ scalar product $\lan f,g\ran =\int_{[-1,1[^2} f(x) g(x) \rmd x$. The space of measures $\mc{M}_x$ is the space of distributions on $\mt^2$ which are continuous (precisely, can be extended continuously) on $C_x$; the space $\mc{M}_x$ can be identified with the space of finite $\mathbb{R}$-valued  \Radon measures on $\mt^2$ and with the quotient space of finite $\mathbb{R}$-valued  \Radon measures on $[-1,1]^2$ under the map $\pi$, via the $L^2_x$ scalar product:
\begin{align*}
\lan f, \mu\ran = \int_{[-1,1[^2} f(x) \mu(\rmd x),\quad \forall f\in C_x.
\end{align*}
For $s>0$  and $1<p<\infty$, denoting by  $p^\ast$ the conjugate exponent of $p$, the space $\nH^{-s,p^\ast}$ is the space of distributions on $\mt^2$ which can be continuously extended to $\nH^{s,p}$.

The convolution on the torus is understood as
\begin{align*}
f\ast g(x) = \int_{[-1,1[^2}f(y)g(x-y) \rmd y
\end{align*}
for $f$, $g$ periodic functions on $\mr^2$.

We recall here some standard facts on the Green function $G$ of the Laplacian on the zero-mean functions, that is
\begin{align*}
\Delta G(\cdot, y) = \delta_y,\quad \forall y\in \mathbb{T}^2.
\end{align*}

\begin{lemma}\label{lem:Green_function}
The following facts hold true.
\begin{enumerate}
\item The Green function $G$ is translation invariant, that is $$G(x,y)=G(x-y), \mbox{ for all $x,y$},  $$ even, regular outside $0$, with
$$-C^{-1}\log|x|\le G(x)\le -C\log|x| \mbox{ in a neighborhood of $0$}.$$
\item The kernel $K=\nabla^\perp G$ is divergence-free (in the distributional sense), odd, regular outside $0$, with $C^{-1}|x|^{-1}\le |K(x)|\le C|x|^{-1}$ in a neighborhood of $0$.
\item Let $\xi$ be a distribution on $\mt^2$ with zero mean, define $u=K \ast \xi$. Then $\diverg u =0$ and $\xi = \curl  u$.
\item Let $u$ be a vector-valued distribution on $\mt^2$ with zero mean and with $\diverg u =0$, define $\xi= \curl  u$. Then $u=K \ast \xi$.
\item Let $u$ be a vector-valued distribution on $\mt^2$, define $\xi= \curl  u$. Then $\Pi u = K \ast \xi$, where $\Pi$ is the Leray projector on zero-mean divergence-free distributions.
\item Let $\xi$ be a distribution on $\mt^2$ with zero mean, define $u=K \ast \xi$. For any $1<p<\infty$, for any  $n \in \mathbb{Z}$, $\xi \in \nH^{n,p}_x$ if and only if $u \in \nH^{n+1,p}_x$ and  there exist a constant $C=C(n)$ independent of $\xi$ such that
\begin{align*}
C^{-1}\|\xi\|_{\nH^{n,p}_x}\le \|u\|_{\nH^{n+1,p}_x}\le C\|\xi\|_{\nH^{n,p}_x}.
\end{align*}
\end{enumerate}
\end{lemma}

For the proof, we recall the following facts:
\begin{itemize}
\item Any distribution $f$ on $\mt^2$ can be written in Fourier series as $f= \sum_k a_k e^{ik\cdot x}$ (the convergence being when tested against a smooth periodic function), see \cite[Section 9]{Tri1983} and \cite[Section 4.11.1]{Tri1978}.
\item For any integer $s$ and any $1<p<\infty$, the Sobolev space $\nH^{s,p}$ can also be written in terms of Fourier series, that is
\begin{align}
\nH^{s,p}_x = \{f \in \mathcal{D}^\prime_x : \tilde{f}^s := \sum_k a_k (1+|k|^2)^{s/2} e^{ik\cdot x} \in L^p_x \},\label{eq:Sob_Fourier}
\end{align}
with $\|\tilde{f}^s\|_{L^p_x}$ as equivalent norm, see \cite[Section 4.11.1]{Tri1978}. This fact is well-known for $s\ge 0$. We give a sketch of the proof for $s<0$ for completeness. We have to show that the above right-hand side is the dual space of $\nH^{-s,p}_x$. Indeed,  for every distributions $f$ continuous on $\nH^{-s,p^\ast}_x$, it holds
\begin{align*}
|\lan \tilde{\varphi}^{-s}, \tilde{f}^s \ran| = |\lan \varphi , f \ran| \le C\|\varphi\|_{\nH^{-s,p^\ast}_x} \le C'\|\tilde{\varphi}^{-s}\|_{L^{p^\ast}_x}, \quad  \forall \varphi \in \nH^{-s,p^\ast}_x,
\end{align*}
hence $\tilde{f}^s$ belongs to $L^p_x$, so $f$ belongs to the right-hand side of \eqref{eq:Sob_Fourier}.
\item Regularity theory: The Laplacian operator $\Delta$, indended in the sense of distribution, acts multiplying each Fourier coefficient $a_k$ by $|k|^2$. In particular, it is invertible on the subspace of zero-mean distributions and its inverse acts multiplying each Fourier coefficient $a_k$ by $|k|^{-2}1_{k\neq 0}$. It follows that the inverse $\Delta^{-1}$ of the Laplacian (on zero-mean distributions) maps $\nH^{s,p}$ into $\nH^{s+2,p}$, for any integer $s$ and any $1<p<\infty$.
\item Hodge decomposition: if $f$ is a $\mr^2$-valued distributions with $\divv f=0$ and $\curl  f=0$, then $f$ is a constant. Indeed, if $a_k$ are the Fourier coefficients of $f$, we have $a_k \cdot k=0$ and $a_k\cdot k^\perp =0$ for every $k$, therefore $a_k=0$ for every $k\neq 0$.
\end{itemize}

\begin{proof}[Proof of Lemma \ref{lem:Green_function}]
\begin{enumerate}
\item The fact that $G$ is translation-invariant is due to the translation invariant property of the torus: if $\varphi$ is periodic and zero-mean and solves $\Delta \varphi =\delta_0$ in the distributional sense, then $\varphi_y(x):=\varphi(x-y)$ is still periodic and zero-mean and solves $\Delta \varphi= \delta_y$. For the even and regularity property and the bounds, see e.g. \cite[Proposition B.1]{BrzFlaMau2016} and references therein.
\item The fact that $K$ is divergence-free, odd and regular outside $0$ is a consequence of the definition of $K$ and the properties of $G$. For the bounds, see again \cite[Proposition B.1]{BrzFlaMau2016}.
\item The fact that $u$ is divergence-free follows from the same property of $K$. Call $\psi=(-\Delta)^{-1}\xi = -G*\xi$. Then $u=-\nabla^\perp \psi$ and so
\begin{align*}
\curl  u = \partial_{x_1}u^2 - \partial_{x_2}u^1 = -\Delta \psi = \xi,
\end{align*}
where all the computations are intended using test functions.
\item Call $\tilde{u}=K \ast \xi$. We deduce from the previous points that $\curl  (u-\tilde{u}) = 0$ and that $\diverg (u-\tilde{u}) = 0$. From this we conclude that $u-\tilde{u}$ is a constant, therefore is $=0$ as both functions have zero mean.
\item This follows from the previous point, applied to $\Pi u$ in place of $u$.
\item If $u$ is in $\nH^{s+1,p}$ then $\xi=\curl  u$ is in $\nH^{s,p}$. Conversely, if $\xi$ is in $\nH^{s,p}$, then $\psi = (-\Delta)^{-1} \xi$ is in $\nH^{s+2,p}$ and so $u = -\nabla^{-1}\psi$ is in $\nH^{s+1,p}$.
\end{enumerate}
\end{proof}

\section{Measurability}
\label{app-D}

We include here various standard concepts and results about measurability.

We recall the definition of strong, weak, weak-$^\ast$ and Borel measurability for a Banach-space valued map. We are given a $\sigma$-finite measure space $(E,\mathcal{E},\mu)$, a Banach space $V$ and a function $f:E\to V$:
\begin{itemize}
\item we say that $f$ is strongly measurable if it is the pointwise
(everywhere) limit of a sequence of $V$-valued simple measurable
functions (i.e. of the form $\sum_{i=1}^{N}v_{i}1_{A_{i}}$ for $A_{i}$
in $\mathcal{E}$ and $v_{i}$ in $V$);
\item we say that $f$ is weakly measurable if, for every $\varphi$ in
$V^{\ast}$, $x\mapsto\langle f(x),\varphi\rangle_{V,V^{\ast}}$ is measurable;
\item if $V=U^{\ast}$ is the dual space of a Banach space $U$, we say that
$f$ is weakly-$\ast$ measurable iff, for every $\varphi \in U$, $x\mapsto\langle f(x),\varphi\rangle_{V,U}$
is measurable;
\item we say that $f$ is resp. strongly Borel, weakly Borel, weakly-$\ast$ Borel measurable if, for every open set $A$ in
$V$ resp. in the strong, weak, weak-$^\ast$ topology, $f^{-1}(A)$ is in $\mathcal{E}$. We omit strongly/weakly/weakly-$\ast$ when clear.
\end{itemize}

The following result is essentially the  Pettis Measurability Theorem. The present version is a consequence of \cite[Chapter I Propositions 1.9 and 1.10]{VTC87}.

\begin{lemma}\label{lem:equiv_topol}
Assume that $V$ is a separable Banach space. Then the notions of the strong measurability, the weak measurability, the strongly Borel measurability and the weakly Borel measurability coincide. They also coincide with the weak-$^\ast$ measurability and weakly-$\ast$ Borel measurability if in addition $V$ is reflexive.
\end{lemma}

%

We prove here a statement concerning weak-$^\ast$ and weakly-$\ast$ Borel measurability, which applies in particular to $\mc{M}_x=(C_x)^\ast$. We call $\bar{B}_R$ the closed centered ball in $V$ of radius $R$ (in the strong topology).

\begin{lemma}\label{rmk:Borel_weakstar}
Assume that $V=U^{\ast}$ is the dual space of a separable Banach space $U$. Then the notions of weak-$^\ast$ measurability and of weakly-$\ast$ Borel measurability coincide. Moreover, for any sequence $(\varphi_k)_k$ dense in the unit centered ball of $U$, the Borel $\sigma$-algebra associated to the weak-$^\ast$ topology is generated by
\begin{align*}
\lan \cdot,\varphi_k\ran.
\end{align*}
\end{lemma}

\begin{remark}\label{rmk:restriction}
We recall that, if $(E,\mc{E})$ is a measurable space, $\mc{I}$ generates the $\sigma$-algebra $\mc{E}$ and $F$ is a subset of $E$, then the $\sigma$-algebra $\mc{E}\mid_F=\{A\cap F: A\in \mc{E}\}$ on $F$ is the $\sigma$-algebra generated on $F$ by $\mc{I}\mid_F=\{I\cap F: I\in \mc{I}\}$. In particular, the Borel $\sigma$-algebra restricted to a subset $F$ is the Borel $\sigma$-algebra on $F$ (with the topology restricted on $F$) and the previous statement can be extended to subsets of $U$.
\end{remark}

\begin{proof}
We fix the sequence $(\varphi_k)_k$. We call $\mc{B}$ the Borel $\sigma$-algebra associated to the weak-$^\ast$ topology and $\mc{C}$ the $\sigma$-algebra generated by the maps $\lan \cdot,\varphi\ran$ for $\varphi \in U$. Since $\varphi_k$ are dense in the unit centered ball of $U$, $\mc{C}$ is generated by the maps $\lan \cdot,\varphi_k\ran$. We will show that $\mc{B}=\mc{C}$, what implies both statements in the Lemma. Since the maps $\lan \cdot,\varphi\ran$, for $\varphi \in U$, are continuous in the weak-$^\ast$ topology, $\mc{C}\subseteq \mc{B}$. For the converse inclusion, it is enough to show that, for any $R>0$, for any open set $A$ in the weak-$^\ast$ topology, the sets $\bar{B}_R$ and $A\cap \bar{B}_R$ are in $\mc{C}$, where $\bar{B}_R$ is the closed centered ball in $V$ of radius $R$ (in the strong topology). By separability of $U$, we can fix a sequence $(\varphi_k)_k$ which is dense in the unit centered ball of $U$. For any $R>0$, the ball $\bar{B}_R$ is in $\mc{C}$ because the strong norm on $V$ is $\mc{C}$-measurable. Indeed,  it can be written as
\begin{align*}
\|v\| = \sup_{k} |\lan v,\varphi_k\ran|.
\end{align*}
We recall that the weak-$^\ast$ topology, restricted on $\bar{B}_R$ is separable and metrizable, see \cite[Theorem 3.28]{Brezis-2011} for the metrizability and  separability follows from the  compactness, with the distance
\begin{align*}
d(v,v') = \sum_{k} 2^{-k}|\lan v-v',\varphi_k\ran|.
\end{align*}
Now, for every $v$ in $\bar{B}_R$, $d(v,\cdot)$ is $\mc{C}$-measurable, hence any open ball with respect to $d$ is in $\mc{C}$. Moreover, for any open set $A$ in the weak-$^\ast$ topology, $A\cap \bar{B}_R$ can be written as countable union of open balls with respect to $d$, hence $A\cap \bar{B}_R$ is in $\mc{C}$. The proof is complete.
\end{proof}

Now we give a measurability property of the testing against bounded, but not necessarily continuous maps. Here, given a Polish space $X$, $\mc{M}(X)$ is the set of finite \Radon measures on $X$.

\begin{lemma}\label{rmk:bounded_test_Borel}
Let $F:X\to \mr$ be a bounded Borel function on a compact metric space $X$ (in particular $X=\mt^2$). Then the map
\begin{align*}
\Psi_F: \mc{M}(X) \ni \mu \mapsto \int_X F(x) \mu(\rmd x) \in \mr
\end{align*}
is Borel with respect to the weak-$^\ast$ topology on $\mc{M}(X)$.
\end{lemma}

\begin{proof}
If $F$ is continuous, then also $\Psi_F$ is continuous in the weak-$^\ast$ topology, in particular weakly-$\ast$ Borel. If $F=1_A$ is the indicator of an open set $A$ in $\mt^2$, then $1_A$ is the pointwise (everywhere) non-decreasing limit on $\mt^2$ of continuous functions $F_n$; so, by the dominated convergence theorem, $\Psi_{1_A}$ is the pointwise limit of $\Psi_{F_n}$ and so it is also weakly-$\ast$ Borel. For the case of general $F$, we use the monotone class theorem. We consider the set $W$ of Borel functions $F$ on $X$ such that $\Psi_F$ is weakly-$\ast$ Borel. Then $\mc{W}$ contains the indicators of all the open sets, it is a vector space and it is stable under monotone non-decreasing convergence. Indeed, if $(F_n)_n$ is a non-decreasing sequence in $\mc{W}$ converging pointwise to $F$, then, by the dominated convergence theorem, $\Psi_F$ is the pointwise limit of $\Psi_{F_n}$, in particular weakly-$\ast$ Borel, and so $F$ belongs also to $\mc{W}$. Then, by the Monotone Class Theorem, $\mc{W}$ contains all bounded Borel functions $F$ on $\mt^2$, which gives the result.
\end{proof}

%

We recall a classical fact for the product of measures. For a compact metric space $X$, we call $\mc{M}(X)$ the set of finite \Radon measures on $X$, dual to the space $C(X)$ of continuous function on $X$, and, for $M>0$, $\mc{M}_M(X)$ the closed centered ball on $\mc{M}(X)$ of radius $M$.

\begin{lemma}\label{rmk:cont_prod_meas}
For any compact metric space $X$, the map $G:\mc{M}(X)\ni \mu\mapsto \mu\otimes \mu \in \mc{M}(X\times X)$ is Borel with respect to the weak-$^\ast$ topologies. Moreover, for any $M>0$, the map $G$, restricted on $\mc{M}_M(X)$ with values in $\mc{M}_{M^2}(X\times X)$, is continuous with respect to the weak-$^\ast$ topologies.
\end{lemma}

\begin{proof}
For $M>0$, we call $G_M:\mc{M}_M(X) \to\mc{M}_{M^2}(X\times X)$ the map $G$ restricted on $\mc{M}_M(X)$ with values in $\mc{M}_{M^2}(X\times X)$. We start showing the continuity of $G_M$. By metrizability of $\mc{M}_M(X)$ and $\mc{M}_{M^2}(X\times X)$, it is enough to show that, if $(\mu^n)_n$ is a sequence in $\mc{M}_M(X)$ converging weakly-$\ast$ to $\mu$, then $(\mu^n\otimes \mu^n)_n$ converges weakly-$\ast$ to $\mu\otimes \mu$. For every two continuous functions $\varphi$, $\psi$ on $X$, we have
\begin{align*}
\lan \varphi\otimes \psi, \mu^n\otimes \mu^n \ran = \lan \varphi,\mu^n\ran \lan \psi,\mu^n\ran \to \lan \varphi\otimes \psi, \mu\otimes \mu \ran.
\end{align*}
Now the set of all linear combinations of $\varphi\otimes \psi$ for all continuous functions $\varphi$, $\psi$ is a subalgebra of $C(X\times X)$ which separates point, hence, by the Stone-Weierstrass theorem, it is dense in $C(X\times X)$. Then, for any $\phi$ continuous function on $X\times X$, by a standard approximation argument on $\phi$ we get that $(\lan \phi,\mu^n\otimes\mu^n \ran)_n$ converges to $\lan \phi,\mu\otimes\mu \ran$. This shows continuity of the map $G$ restricted to $\mc{M}_M(X)$.

For Borel measurability on the full space, take any open set $A$ in $\mc{M}(X\times X)$, then $G^{-1}(A)$ is the non-decreasing union of $G_M^{-1}(A\cap \mc{M}_{M^2}(X\times X))$ for $M$ in $\mathbb{N}$. By continuity of $G_M$, $G_M^{-1}(A\cap \mc{M}_{M^2}(X\times X))$ is open, hence Borel, in $\mc{M}_M(X)$. Moreover $\mc{M}_M(X)$ is itself a Borel set in $\mc{M}(X)$. Indeed the closed centered ball $\bar{B}_R$ in a dual space $V=U^\ast$ is Borel, as shown in the proof of Lemma \ref{rmk:Borel_weakstar}. So $G_M^{-1}(A\cap \mc{M}_{M^2}(X\times X))$ is Borel in $\mc{M}(X)$. Therefore $A$ is Borel in $\mc{M}(X)$. The proof is complete.
\end{proof}

We conclude on measurability of the $H^h$ norms:

\begin{remark}\label{rmk:Borel_norm}
For any fixed integer $h$, the $H^h_x$ norm can be written as supremum of $|\lan \cdot,\varphi\ran|$ over a set $D$ of $\varphi \in C_x$, with $D$ countable and dense in $H^h_x$. Therefore the $H^h_x$ norm is a lower semi-continuous function and Borel function on $(\mc{M}_x,w^\ast )$.

The $C_t^\alpha(H^h_x)$ norm can be written as
\begin{align*}
\|f\|_{C^\alpha_t(H^h_x)}= \sup_{t\in \mathbb{Q}\cap[0,T]}\|f_t\|_{H^h_x} +\sup_{s,t\in \mathbb{Q}\cap[0,T],s<t} \frac{\|f_t-f_s\|_{H^h_x}}{|t-s|^\alpha}
\end{align*}
(note the supremum over a countable set of times). Therefore, for any fixed $M>0$, the $C_t(H^h_x)$ norm is a lower semi-continuous function, in particular a Borel function, on $C_t(\mc{M}_{x,M},w^\ast )$.
\end{remark}

\end{appendices}
\noindent
\textbf{Acknowledgement.} We would like to thank to James-Michael Leahy for pointing out the presence of the constant $\gamma$ in the equation \eqref{eq:stochEulervel_intro} for the velocity, and to  Jasper Hoeksema and Oliver Tse for pointing out a mistake in Lemma \ref{rmk:bounded_test_Borel}, in a previous draft of this paper, and a suggestion  how to correct it. This work was undertaken mostly when M.M.~was at the University of York, supported by the Royal Society via the Newton International Fellowship NF170448 ``Stochastic Euler Equations and the Kraichnan model''. Finally, we would like to thank Philippe Serfati for pointing out several references on deterministic Euler equations.

\bibliographystyle{alpha}
\bibliography{my_bib7}

\end{document}